\documentclass[a4paper,11pt,reqno,twoside]{amsart}

\usepackage{amsmath}
\usepackage{amsfonts}
\usepackage{amssymb}
\usepackage{amsthm}
\usepackage{color}
\usepackage{ifpdf}
\usepackage{array}
\usepackage{url}
\usepackage{multirow,bigdelim}
\usepackage{ascmac}
\usepackage{mathrsfs}
\usepackage[all]{xy} 
\usepackage{graphicx}
\usepackage{ulem}

\newtheorem{theorem}{Theorem}[section]
\newtheorem{proposition}[theorem]{Proposition}
\newtheorem{lemma}[theorem]{Lemma}
\newtheorem{corollary}[theorem]{Corollary}

\numberwithin{equation}{section}

\newcommand*{\C}{\mathbb{C}}
\newcommand*{\R}{\mathbb{R}}
\newcommand*{\Z}{\mathbb{Z}}

\newcommand*{\D}{\mathscr{D}}

\newcommand{\comment}[1]{}
\title[Weil's quadratic form via the screw function]%
      {Weil's quadratic form \\ via the screw function} 
\author[M. Suzuki]{Masatoshi Suzuki}
\date{Version of \today}
\subjclass[2020]{
11M26 
42A82 
46E22 
47B25 
}
\keywords{
Riemann Hypothesis;
Weil quadratic form;
screw function;
de Branges spaces;
Hilbert--P\'olya operator
}
\AtBeginDocument{%
\begin{abstract}
We establish a unified operator-theoretic framework 
for understanding the results on the Weil quadratic form obtained by 
Yoshida (1992), Bombieri (2001, 2003), Connes--Consani (2023), and 
Connes--Consani--Moscovici (2025+) from 
the perspective of the screw function introduced in Suzuki (2023). 
An advantage of the approach via the screw function is that 
it provides a method to study the Weil quadratic form, 
which is originally defined in terms of distributions, 
by means of continuous functions and
brings the theory of de Branges spaces into the analysis of the Weil quadratic form 
via the Fourier transform.
Based on this framework, we formulate a conjecture stating that a self-adjoint operator whose 
eigenvalues are the imaginary parts of the nontrivial zeros of the Riemann zeta function 
can be obtained as the limit, 
as $a \to \infty$, of self-adjoint operators arising from nonlocal realizations of 
the first-order differential operator on the finite interval $[-a,a]$. 
All these results are obtained without assuming the Riemann Hypothesis. 
This conjecture may be compared with the limit formula for the Riemann zeta function expressed 
in terms of zeta-regularized products proposed by Connes, Consani, and Moscovici, and 
it sheds new light on the spectral-theoretic interpretation of the nontrivial zeros of the Riemann zeta function.
\end{abstract}
\maketitle
}
\begin{document}

%
\section{Introduction} \label{section_1}
%

\subsection{Historical Background and Motivation} 

For test functions $f$, the Weil functional $f \mapsto W(f)$ is defined by
\begin{equation*} 
\aligned 
W(f)
& := \int_{-\infty}^{\infty} f(x) (e^{x/2}+e^{-x/2}) dx 
 - \sum_{n=1}^{\infty} \frac{\Lambda(n)}{\sqrt{n}} f(\log n) 
- \sum_{n=1}^{\infty} \frac{\Lambda(n)}{\sqrt{n}} f(-\log n)  \\
& \quad 
- (\log 4\pi + C_0) f(0)
- \int_{0}^{\infty} \left\{ f(x) + f(-x)-2e^{-x/2} f(0)\right\} \frac{e^{x/2}dx}{e^{x}-e^{-x}},
\endaligned 
\end{equation*}
where $\Lambda(n)$ denotes the von Mangoldt function, 
defined as $\Lambda(n)=\log p$ if $n=p^k$ with $k \in \Z_{>0}$,  
and $\Lambda(n)=0$ otherwise,  
and $C_0$ is the Euler--Mascheroni constant. 
Although the precise class of test functions will be specified as needed later, 
our primary object of interest is the symmetric (or Hermitian) quadratic form
\begin{equation*} 
Q_W(v_1,v_2):=W(v_1 \ast \widetilde{v_2}), \quad Q_W(v):=Q_W(v,v),
\end{equation*}
where 
\begin{equation*} 
(v_1\ast v_2)(x):=\int_{-\infty}^{\infty} v_1(y)v_2(x-y) \, dy 
\quad \text{and} \quad 
\widetilde{v}(x) := \overline{v(-x)}. 
\end{equation*} 
The Riemann Hypothesis (RH) states that 
all nontrivial zeros of the Riemann zeta function $\zeta(s)$ 
lie on the critical line $\Re(s)=1/2$. 
A fundamental result due to Weil \cite{We52} states that 
RH is equivalent to the condition that
\begin{equation*} 
Q_W(v) \geq 0 \quad \text{for all } v \in C_c^\infty(\R),
\end{equation*}
a property known as Weil's positivity criterion for RH 
(although Weil did not formulate the criterion in terms of compactly supported smooth test functions; see \cite{We52,We72} and \cite[Section 3.2]{Su23}). 
In the commentary to his Collected Works,
Weil emphasized the distributional nature of this criterion
and suggested that it deserved further investigation by analysts. 
Despite this invitation, 
the functional $W$ and the quadratic form $Q_W$ have not been extensively studied, 
with notable exceptions such as the works of 
Yoshida \cite{Yo92}, Bombieri \cite{Bo01, Bo03}, 
Connes and Consani \cite{CC23}, and 
Connes, Consani, and Moscovici \cite{CCM25}.

The study of the localization of the Weil quadratic form $Q_W$ was pioneered by Yoshida \cite{Yo92}, 
who considered the positive definiteness of $Q_W$ on the spaces
$
C(a):=\{v\in C_c^\infty(\R)\mid {\rm supp}(v)\subset[-a,a]\}
$
and observed that RH is equivalent to its positive definiteness for every $a>0$.
In the present paper, we use the standard notation
\[
C_c^\infty(-a,a)
:=
\{v\in C_c^\infty(\R)\mid {\rm supp}(v)\subset(-a,a)\}.
\]
Thus $C_c^\infty(-a,a)\subset C(a)$. 
The difference at the endpoints will not affect the arguments below. 
An interesting question is whether, for a fixed $a>0$, 
the positivity of $Q_W$ on such localized test functions has a direct arithmetic interpretation,
for instance in terms of the distribution of primes.
To the best of the author's knowledge, 
no such interpretation is currently known.
In \cite[Proposition 1]{Yo92}, 
Yoshida further demonstrated that the condition $Q_W(v) \geq 0$ 
for all non-zero odd functions $v \in C_c^\infty(\R)$ implies RH, 
while a similar condition for all non-zero even functions implies RH except for possible real zeros. 
Here Yoshida works more generally with Dedekind zeta functions. 
For the Riemann zeta function, the possibility of such exceptional real zeros 
can be excluded by elementary arguments. 
The statement for even functions is therefore relevant in the more general setting of 
Dedekind zeta functions.
Crucially, he initiated the variational study of $Q_W$ 
by considering the infimum of the Rayleigh quotient $Q_W(v)/\|v\|_{L^2}^2$ on the localized space 
\[
K(a):=\{v~|~v(x)=f(x)\mathbf{1}_{[-a,a]}(x)~\text{for some $f \in C^\infty(\R)$ which has period $2a$}\}, 
\] 
proving that it is positive for sufficiently small $a > 0$ \cite[Lemma 2]{Yo92} 
or for appropriate subspaces $K_N(a)$ of finite codimension \cite[Lemma 3]{Yo92}. 
The non-degeneracy of $Q_W$ on the ``completion'' $\widehat{K(a)}~(\hookrightarrow L^2(-a,a))$ 
was shown to be equivalent to RH in \cite[Theorem 2]{Yo92}. 
Note that the hat in $\widehat{K(a)}$ does not denote the Fourier transform.

Building on this foundation, 
Bombieri \cite{Bo01, Bo03} addressed the minimization of the Rayleigh quotient
on a Sobolev space \cite[Problem 1]{Bo01}
(and on a related Sobolev space in \cite[Problem B]{Bo03})
and on $L^2(\mathcal{E})$ for a finite union of intervals $\mathcal{E}$ \cite[Problem 2]{Bo01}
(and \cite[Problem A]{Bo03}). 
While the existence of a minimizer for the lower bound on $L^2(\mathcal{E})$ 
was established in \cite[Theorem 3]{Bo01} (and \cite[Theorem 4.3]{Bo03}), 
the subsequent argument concerning the continuity of this lower bound with respect to $a$ 
remains an analytically delicate issue. Specifically, although it was claimed in \cite[Theorem 5]{Bo01} that 
the lower bound varies continuously when the minimization is restricted to the subspaces of 
either even or odd functions, this assertion calls for a more careful analysis. 
In view of Yoshida's result, such continuity would essentially reproduce the equivalence 
between the failure of RH and the existence of 
a degenerating element $v \in \mathfrak{D}(Q_W)$ 
satisfying $Q_W(v)=0$, offering a simplification of the original proof that relied on explicit formulas. 
See the discussion following Theorem~\ref{thm_3} for further details.

More recently, a rigorous operator-theoretic refinement has been provided by Connes and Consani \cite{CC23} and Connes, Consani, and Moscovici \cite{CCM25}. 
They established that the localized closed symmetric form 
\[
Q_W^{\,a}:=Q_W|_{L^2(-a,a)}
\] 
is lower bounded and lower semicontinuous \cite[Section 2]{CC23}, 
leading to the construction of a canonical, 
densely defined self-adjoint operator $A_a$ on $L^2(-a, a)$ 
such that 
\begin{equation} \label{EQ_101}
Q_W^{\,a}(v)=\langle A_a v, v \rangle_{L^2} 
\end{equation}
for $v \in \mathfrak{D}(A_a) \subset \mathfrak{D}(Q_W^{\,a}) \subset L^2(-a,a)$ \cite[Section 3]{CCM25} 
(here we reformulate their results, originally stated on $[1/\lambda, \lambda]$, 
in terms of the variable $a=\log\lambda$), 
where the bracket is the standard $L^2$ inner product 
\[
\langle v_1,v_2 \rangle_{L^2} =\int_{-\infty}^{\infty}v_1(x)\overline{v_2(x)}dx. 
\]
This operator $A_a$, 
which can be viewed as a refinement of Bombieri's Lagrangian 
(\cite[Lemma 1]{Bo01}, \cite[Sections 5--6]{Bo03}), 
possesses a discrete lower bounded spectrum \cite[Theorem 3.6]{CCM25} and 
a corresponding ground state \cite[Corollary 3.7]{CCM25}. 
Furthermore, it was shown in \cite[Theorem 5.10]{CCM25} that, 
when $Q_W^{\, a}$ is restricted to a certain finite-dimensional subspace of $L^2(-a, a)$, 
the Fourier transform of the eigenfunction corresponding to the lowest eigenvalue 
possesses only real zeros.
\medskip

First, we aim to unify these various results through 
the framework of the screw function associated with $\zeta(s)$, 
which was introduced in \cite{Su23}. 
This approach not only provides transparent re-derivations of previous results, 
but also leads to more refined results presented below. 
A key advantage of this screw function approach is that it enables us to understand Weil's quadratic form through continuous (and well-behaved) functions, 
thereby avoiding the delicate distribution-theoretic aspects of Weil's original formulation.
The connection between the Weil quadratic form and the screw function, 
already emphasized in \cite[Proposition 3.1]{Su23}, 
was expected to provide new insight into the study of the Weil quadratic form. 
Indeed, \cite{Su23} reinterpreted some of the results 
by Yoshida~\cite{Yo92}, Bombieri~\cite{Bo01, Bo03}, and Connes and Consani~\cite{CC23} 
from the viewpoint of screw functions. 
However, these earlier works had not yet been placed within a fully unified framework. 
Since then, the theoretical foundations of the screw-function approach to zeta functions 
have been substantially developed. 

In particular, \cite{Su25} established a completely new picture of the Hilbert space obtained as 
the completion of $C_c^\infty(\R)$ with respect to $Q_W$ under RH. 
This was made possible by the particular analytic properties of the screw function as 
an integral kernel, such as positivity and smoothness.
Our main objective is to understand what new insights about $Q_W$
can be obtained by localizing the global framework of \cite{Su25}
to the interval $[-a,a]$,
thereby returning to Yoshida's original viewpoint. 
This line of investigation was further motivated by recent results in 
\cite[Theorem 5.10, Sections 7--8]{CCM25}.
Let
\[
\xi(s):=s(s-1)\pi^{-s/2}\Gamma(s/2)\zeta(s)
\]
be the Riemann $\xi$-function, where $\Gamma(s)$ 
is the usual gamma function. Let $v_a$ be the eigenfunction corresponding to 
the lowest eigenvalue of $A_a$ (denoted by $\xi_\lambda$ with $a=\log\lambda$ in their notation). Denote by
\[
\widehat{f}(z):=\int_{-\infty}^{\infty}f(x)\,e^{izx}dx
\]
the Fourier transform of $f$. 
Motivated by the Hilbert--P\'olya operator constructed in \cite[Section 6]{Su25}, 
we next formulate an analogue of the conjectural limiting formula 
discussed in \cite[Section 7]{CCM25},
\begin{equation} \label{EQ_102}
\lim_{a \to \infty} c_a \widehat{v_a}(z)=\xi(1/2+iz).
\end{equation}
This line of investigation led to a unified perspective on the earlier works mentioned above
and ultimately to the conjectural limit formula
\eqref{EQ_112} in Corollary~\ref{cor_6}. 

\subsection{Main Results} 

Before presenting the main results, 
we emphasize that none of them depends on RH. 
The first step in the study of the Weil quadratic form via screw functions 
is to obtain an explicit description of the self-adjoint operator $A_a$ in \eqref{EQ_101}. 
While the existence of $A_a$ was established abstractly 
in \cite[Section 3]{CCM25} via the theory of quadratic forms, 
the use of screw functions provides a constructive and explicit description of $A_a$.
\medskip

Formally, $A_a$ may be viewed as an integral operator with difference kernel $k(x-y)$, 
where 
$k(t)=\sum_{\xi(1/2+i\gamma)=0} e^{-i\gamma t}$ 
is understood as a distribution.
Integrating $-k$ twice leads to a continuous function. 
Applying Weil's explicit formula to it yields the following continuous function.

Let $g(t)$ be a continuous real-valued even function on $\R$,
defined by
\begin{equation} \label{EQ_103}
\aligned 
g(t) 
& = -4(e^{t/2}+e^{-t/2}-2) + \sum_{n \leq \exp(|t|)} \frac{\Lambda(n)}{\sqrt{n}}(|t|-\log n)  \\
& \quad - \frac{|t|}{2}( \psi(1/4) - \log \pi )
- \frac{1}{4}\left( \Phi(1,2,1/4) - e^{-|t|/2}\Phi(e^{-2|t|},2,1/4) \right),  
\endaligned 
\end{equation}
where $\psi(s)$ is the digamma function 
and $\Phi(z,s,a) = \sum_{n=0}^{\infty} (n+a)^{-s}z^n$ 
is the Hurwitz--Lerch zeta function. 
A continuous function $g$ on $\R$ satisfying $g(t)=\overline{g(-t)}$
 is called a screw function on the real line 
if the kernel $g(t-u)-g(t)-g(-u)+g(0)$ is nonnegative definite on $\R$. 
As shown in \cite[Theorem 1.2]{Su23}, 
the above function $g$ is a screw function 
in the sense of Krein and Langer~\cite[Section 5]{KrLa14}
if and only if RH holds. 
For this reason, we shall refer to it as the screw function associated with $\zeta(s)$. 
We define the integral operator $G$ by
\begin{equation} \label{EQ_104}
(Gu)(x):=\int_{-\infty}^{\infty}g(x-y)u(y)\,dy. 
\end{equation}
The integral on the right-hand side converges absolutely for $u \in C_c^\infty(\R)$, 
since $g$ is continuous. 
However, as noted immediately after Theorem 1.6 in \cite{Su23}, 
$g(t)$ does not tend to zero as $|t| \to \infty$, 
and hence the integral does not necessarily converge for a general $u \in L^2(\R)$.  
Within the theory of screw functions, 
it would be more natural to consider the integral operator $\widetilde{G}$ 
with the kernel $g(t-u)-g(t)-g(-u)+g(0)$ (cf. Section~\ref{section_8}); 
nevertheless, for applications to $Q_W$, 
the simpler operator $G$ is sufficient. 

Throughout this paper, for each $a>0$, 
we view $L^2(-a,a)$ as a closed subspace of $L^2(\R)$ by extending functions to be zero outside $(-a,a)$. 
Let $L_0^2(-a,a)$ denote the closed subspace of $L^2(-a,a)$ 
consisting of functions $u$ satisfying $\widehat{u}(0)=\int_{-a}^{a} u(x)dx=0$. 
We denote by $P_a$ the orthogonal projection from $L^2(\R)$ onto $L_0^2(-a,a)$  
(see \eqref{EQ_201}), and define
\begin{equation} \label{EQ_105}
G_a:=P_a GP_a : L_0^2(-a,a) \to L_0^2(-a,a). 
\end{equation}
The self-adjointness of the operator $G_a$ 
follows immediately from the fact that $g$ is real-valued and even. 
By \cite[Proposition 3.1]{Su23}, 
the integral operator $G_a$ associated with the screw function 
is directly related to Weil's quadratic form $Q_W$ 
via differentiation of test functions. 
Let $D=i\,d/dx$ denote the differentiation operator on $L^2(-a,a)$ 
with Dirichlet boundary conditions, and let $D^\ast$ be its adjoint. We then define
\begin{equation} \label{EQ_106}
B_a:=D^\ast G_a D : L^2(-a,a) \to L^2(-a,a), \quad \mathfrak{D}(B_a)=H_0^1(-a,a). 
\end{equation}
This operator is symmetric but fails to be self-adjoint. 
Then, the operator $A_a$ in \eqref{EQ_101} can be characterized as follows.

\begin{theorem} \label{thm_1}
For each \(a>0\), the self-adjoint operator \(A_a\) defined by
\eqref{EQ_101} is the Friedrichs extension of the symmetric
operator \(B_a\) defined by \eqref{EQ_106}.
\end{theorem}

This result provides a concrete operator-theoretic realization of 
the spectral interpretation of $Q_W^{\,a}$ established in \cite{CCM25} 
through the use of screw functions, 
thereby clarifying the connection between this spectral interpretation 
and the classical variational approach of Yoshida~\cite{Yo92}. 

Moreover, as shown in \cite[Theorem 3.6]{CCM25}, 
for each $a>0$, the spectrum of the self-adjoint operator $A_a$ is bounded from below and discrete, 
with $+\infty$ as its only accumulation point. 
In particular, the largest lower bound of the spectrum of $A_a$, denoted by $\lambda_a$, 
is an eigenvalue. Therefore, 
there exists a nonzero function $v \in \mathfrak{D}(A_a)$ attaining the infimum
\begin{equation} \label{EQ_107}
\lambda_a=\inf_{0\ne v \in \mathfrak{D}(Q_W^{\,a})}\frac{Q_W^{\,a}(v)}{\|v\|_{L^2}^2}.
\end{equation}
Furthermore, the domain $\mathfrak{D}(A_a)$ 
is strictly larger than $\mathfrak{D}(B_a)=H_0^1(-a,a)$ 
and contains functions such as constants. 
As will be shown later (Lemma \ref{lem_3_1}), 
we have
\begin{equation} \label{EQ_108}
Q_W^{\,a}(v)=\langle B_a v, v \rangle_{L^2}\quad \text{for} \quad v \in H_0^1(-a,a).
\end{equation}
Since the form-norm closure of $C_c^\infty(-a,a)~(\subset H_0^1(-a,a))$ with respect to $Q_W^{\,a}$ 
contains the core of the quadratic form $Q_W^{\,a}$ as in the proof of Theorem~\ref{thm_1}, 
we obtain the following result.

\begin{corollary} \label{cor_2} 
For each $a>0$, any function attaining the infimum in \eqref{EQ_107}
belongs to the closure of $C_c^\infty(-a,a)$ with respect to the form norm
associated with $Q_W^{\,a}$.
In particular, $\lambda_a$ is the infimum of the Rayleigh quotient
\begin{equation*} 
\frac{\langle B_a v, v \rangle_{L^2}}{\|v\|_{L^2}^2}
\end{equation*}
over $C_c^\infty(-a,a)$.
\end{corollary}

This is analogous to \cite[Problem 2]{Bo01}
(and \cite[Problem A]{Bo03}) and \cite[Theorem 3.6]{CCM25},
but here the class over which the infimum is taken is made explicit.
For the relation between the Rayleigh quotient in
\cite[Problem 1]{Bo01} (and \cite[Problem B]{Bo03})
and the operator $G_a$, see Section~\ref{Bombieri_Problem_1}.
\medskip

One expects that the eigenvalues of the self-adjoint operator $A_a$
vary continuously in $a$. 
However, because $A_a$ is unbounded and its domain is difficult to describe explicitly,
it is not straightforward to prove such continuity.
Indeed, in \cite[Theorem 5]{Bo01}
(and \cite[Theorem 4.4]{Bo03}),
analogues of the infimum in \eqref{EQ_107}
are considered on the spaces of even and odd functions,
and their continuity with respect to $a$ is asserted,
but the details of the proof are not fully provided there.

In contrast, by exploiting the asymptotic expansion
\eqref{EQ_202} of the screw function near the origin,
which involves a logarithmic singularity,
one can prove the continuity of the lowest eigenvalue $\lambda_a$
without imposing any parity restriction.

\begin{theorem} \label{thm_3}
The lowest eigenvalue $\lambda_a$ is continuous in $a$.
\end{theorem}

Since the continuity of $\lambda_a$ can be established without assuming RH,
Theorem~\ref{thm_3} immediately yields, as a corollary,
another proof of Yoshida's result \cite{Yo92} that
RH is equivalent to the nondegeneracy of $Q_W^{\,a}$ for every $a>0$.
Indeed, the failure of RH is equivalent to the existence of some $a>0$
for which $\lambda_a<0$.
Since $\lambda_a>0$ for sufficiently small $a>0$ \cite[Lemma 2]{Yo92},
it follows that if RH is false, then $Q_W^{\,a}$ must be degenerate
for some value of $a$ by continuity of $\lambda_a$. 
\medskip

We now turn to the eigenfunction corresponding to the lowest eigenvalue,
following the approach of \cite{CCM25}.
Roughly speaking,
\cite[Theorem 5.10]{CCM25}
(see also~\cite[Section 7]{CCM25})
shows that if the lowest eigenvalue $\lambda_a$ is simple
and its eigenfunction
$v_a \in \mathfrak{D}(A_a)$ is even,
then all zeros of the Fourier transform $\widehat{v_a}(z)$
(corresponding to $\widehat{\xi_\lambda}(z)$ in~\cite{CCM25})
are real.
More precisely, they study restrictions of $Q_W^{\,a}$
to suitable finite-dimensional subspaces of $L^2(-a,a)$,
obtaining the corresponding reality result at each finite-dimensional level.
In our framework, however,
the simplicity of $\lambda_a$
and the evenness of the corresponding eigenfunction
follow naturally for sufficiently small $a$
by once again exploiting the asymptotic expansion
\eqref{EQ_202} of the screw function.

\begin{theorem} \label{thm_4}
For sufficiently small $a>0$,
the lowest eigenvalue $\lambda_a$ is positive,
simple, and satisfies
\[
\lambda_a  = \log \frac{1}{a} + \mu_1- \log(2\pi)+\psi(2) - 1 + O(a)
\]
as $a\to0+$, for some constant $\mu_1>0$.
Furthermore, the corresponding eigenfunction is even.
\end{theorem}

Let us return to the eigenfunction $v_a$
corresponding to the lowest eigenvalue. 
The fact that all zeros of its Fourier transform $\widehat{v_a}(z)$ are real
follows from the property that it serves as the characteristic function
(given by a zeta-regularized product)
of a self-adjoint perturbation of the differential operator $d/dx$
with periodic boundary conditions
\cite[Theorem 5.10]{CCM25}.
However, since such a differential operator cannot be defined directly
as an operator on $L^2(-a,a)$,
the authors of \cite{CCM25} instead consider restrictions of $Q_W^{\,a}$
to finite-dimensional subspaces $V$.
Consequently, what is actually established there is not directly
the reality of the zeros of $\widehat{v_a}(z)$ itself,
but rather the reality of the zeros of the Fourier transform
of a function minimizing the Rayleigh quotient of $Q_W^{\,a}|_V$.

\medskip

Instead of considering perturbations of the differential operator $d/dx$
with periodic boundary conditions,
we construct an analogue of $\widehat{v_a}(z)$
by considering self-adjoint extensions of the minimal operator
$\D_a$
in a Hilbert space different from the usual $L^2$ space.
Choose $\lambda$ such that $\lambda_a>\lambda$, and define
$T_a=T_{a,\lambda}:=A_a-\lambda I$.
Since $T_a$ is positive,
\begin{equation} \label{EQ_109}
\| v \|_{T_a}^2 := \langle T_a v, v \rangle_{L^2}=Q_W^{\,a}(v) - \lambda\|v\|_{L^2}^2
\end{equation}
defines a norm on $C_c^\infty(-a,a)$.
Let $\mathcal{H}(T_a)$ denote the corresponding completion.
Then $C_c^\infty(-a,a)$ embeds injectively into $\mathcal{H}(T_a)$.
We therefore consider the differentiation operator
$\mathscr{D}_a:=i\,d/dx$ on $\mathcal{H}(T_a)$ with domain
\begin{equation} \label{EQ_110}
\mathfrak{D}(\mathscr{D}_a):=C_c^\infty(-a,a).
\end{equation}
It turns out that $\mathscr{D}_a$ is a symmetric operator
with deficiency indices $(1,1)$
(Lemmas~\ref{lem_7_1} and \ref{lem_6_2}).
Hence $\mathscr{D}_a$ admits a family of self-adjoint extensions
$\overline{\mathscr{D}}_{a,\theta}$ parametrized by $\theta\in[0,2\pi)$.
A function arising from the boundary form of the adjoint operator
$\D_a^\ast$ provides a counterpart to $\widehat{v_a}(z)$ in \cite{CCM25}, as follows. 

\begin{theorem} \label{thm_5}
Let $a>0$ and $\theta\in[0,2\pi)$.
Let $v_\pm(a,x)$ be eigenfunctions of the adjoint operator $\D_a^\ast$
corresponding respectively to the eigenvalues $\pm i$,
normalized so that
$\|v_+(a,\cdot)\|_{T_a}=\|v_-(a,\cdot)\|_{T_a}$.
Then the function
\begin{equation} \label{EQ_111}
W(a, \theta; z):=
(z - i) \int_{-a}^a  v_{+}(a, x) e^{iz x} dx
+ e^{i\theta} (z + i) \int_{-a}^a  v_{-}(a,x) e^{iz x} dx
\end{equation}
is  entire in $z$.
The eigenvalues of the self-adjoint operator
$\overline{\mathscr{D}}_{a,\theta}$
are precisely the zeros of $W(a,\theta;z)$. 
Furthermore, all zeros of $W(a,\theta;z)$ are real and simple. 
\end{theorem}

The equation $W(a,\theta;z)=0$ is equivalent to
\[
e^{i\theta}
=
- \left. \left(
 (z-i)\int_{-a}^a v_+(x)e^{izx}\,dx  \right) \right\slash
\left( (z+i)\int_{-a}^a v_-(x)e^{izx}\,dx \right). 
\]
This is analogous to the equation
$e^{i\theta}=\exp(iza)/\exp(-iza)$,
which characterizes the eigenvalues of the self-adjoint extension
$\partial_\theta$ of the minimal operator
$\partial=i\,d/dx$ on $L^2(-a,a)$
through the boundary values of its eigenfunctions.

\medskip

Let us compare Theorem~\ref{thm_5}
with \cite[Theorem 5.10]{CCM25}.
First, the two results are parallel in that both establish
the reality of the zeros of a certain function
through its relation to the self-adjointness of a differential operator.
However, whereas the result in \cite{CCM25} requires strong assumptions,
such as the simplicity of $\lambda_a$
and the evenness of the corresponding eigenfunction,
our Theorem~\ref{thm_5} has the advantage of being proved unconditionally. 
In fact, its proof does not require detailed information on the arithmetic terms 
appearing in the Weil form $Q_W^{\,a}$. 
It relies essentially only on the fact that 
the prime contribution to $Q_W^{\,a}$ involves only finitely many primes for fixed $a>0$.  
Moreover, they obtain a determinant representation
of $\widehat{v_a}(z)$ in terms of a zeta-regularized determinant
(by restricting to finite-dimensional subspaces),
thereby interpreting $\widehat{v_a}(z)$ as a characteristic function.
By contrast, our expression \eqref{EQ_111}
arises from a Lagrangian condition.

A further difference appears in the approach to RH,
as discussed in \cite[Sections 7 and 8]{CCM25}.
While their function $\widehat{v_a}(z)$ is expected
to approximate $\xi(1/2-iz)$ in the limit $a\to\infty$ 
through the conjectural formula \eqref{EQ_102}, 
our function $W(a,\theta;z)$ is expected,
as explained below,
to approximate the reciprocal of one plus its logarithmic derivative.

\begin{corollary} \label{cor_6}
If one can choose $\lambda=\lambda(a)<\lambda_a$, $\theta=\theta(a)$, and $\phi(a,z)$ 
for all sufficiently large $a>0$ 
such that $\phi(a,z)$ is analytic in the upper half-plane $\{z \in \C \mid \Im z>0\}$ 
and 
\begin{equation} \label{EQ_112}
\lim_{a \to \infty} e^{\phi(a,z)}\, W(a,\theta;z)
= 
\frac{\xi(1/2-iz)}{\xi(1/2-iz)+\xi'(1/2-iz)}
\end{equation}
holds uniformly on every compact subset of the upper half-plane,
then RH holds. 
\end{corollary}

The eigenfunction $v_{a,z}$ of the adjoint operator $\D_a^\ast$ 
depends on the parameter $\lambda$ defining $T_a$, 
and hence $W(a,\theta;z)$ also depends on $\lambda$. 
Under RH, one has $A_a>0$, so that $\lambda=0$
may be chosen in $T_a$ for every $a>0$.
The limit formula \eqref{EQ_112}
is formulated with the situation $T_a=A_a$
(i.e. $\lambda=0$) in mind. 
The reason why the limit formula \eqref{EQ_112}
is expected to hold will be discussed in Section~\ref{section_7} 
under the assumption $A_a>0$ for every $a>0$.
In brief, it is motivated by the determination of the structure of
$\mathcal{H}(A_\infty)
:=\mathcal{H}(T_{\infty,0})$
under RH in \cite[Theorem 1.1]{Su25},
together with the construction of the Hilbert--P{\'o}lya operator
described in \cite[Section 6]{Su25}. 
More specifically, Section~\ref{section_7} gives further heuristic considerations suggesting \eqref{EQ_112}, based on structural observations concerning the boundary forms of $\D_a^\ast$ and 
reproducing kernels arising in the analysis of $\mathcal{H}(A_\infty)$. 
In view of these considerations, 
we expect that a proof of \eqref{EQ_112} based on the mechanism
described in Section~\ref{section_7} would require taking
$\lambda=0$ for all sufficiently large $a>0$.
For this choice to be admissible, one must have $\lambda_a>0$
for all sufficiently large $a>0$.
Since $\lambda_a$ is non-increasing in $a$ by Corollary~\ref{cor_2},
this is equivalent to $\lambda_a>0$ for every $a>0$,
and hence to RH by Weil's criterion.
Thus, if this expectation is correct,
Corollary~\ref{cor_6} would not provide a simpler route to RH
than Weil's positivity criterion itself.

A direct computation of the eigenfunctions shows that
every eigenvalue of $\overline{\mathscr{D}}_{a,\theta}$
has multiplicity one (see Section \ref{section_6}).
The right-hand side of \eqref{EQ_112}
is consistent with this simplicity of the spectrum. 
The form of the right-hand side of \eqref{EQ_112} 
suggests that $\theta=\pi$ may be the most natural choice, for reasons discussed in Section~\ref{section_7_8}.
The exponential factor $\exp(\phi(a,z))$ is included
to allow for a possible normalization to compensate for singularities arising in the limit process 
(see Section \ref{section_7_9} for details). 
\medskip

Since the reality of the zeros of $W(a,\theta;z)=W(a,\lambda, \theta;z)$
for each $a>0$ has already been established in
Theorem~\ref{thm_5},
it would suffice to prove \eqref{EQ_112}
as an identity of complex functions in order to deduce RH. 
As explained in Section~\ref{section_8},  
the eigenfunctions $v_\pm(a,x)$ satisfy
Fredholm integral equations of the first kind
with continuous kernels; see \eqref{eq_260725_3}.
Consequently, the conjectural limit formula
\eqref{EQ_112} in Corollary~\ref{cor_6}
may be formulated as a purely analytic statement. 
\medskip

The organization of this paper is as follows.
In Section~\ref{section_2},
we study the asymptotic expansion of the screw function $g$
near the origin and derive from it several formulas for the quadratic form
$\langle B_a v,v\rangle$.
The latter formulas will be used repeatedly in the proofs
of the main results.
Using the results of Section~\ref{section_2},
we prove Theorem~\ref{thm_1} in Section~\ref{section_3}
and Theorem~\ref{thm_3} in Section~\ref{section_4}.
In both proofs,
the compactness of the embedding
$\mathfrak{D}(Q_W^{\,a}) \hookrightarrow L^2(-a,a)$
plays an essential role.
Theorem~\ref{thm_4} is proved in Section~\ref{section_5}.
In addition to the results of Section~\ref{section_2},
the proof relies on the theory of Dirichlet forms. 
In Section~\ref{section_6},
we prove Theorem~\ref{thm_5}
using von Neumann's theory of self-adjoint extensions.
Section~\ref{section_7}
explains the motivation for the limit formula in
Corollary~\ref{cor_6}
by relating it to the results of \cite{Su25}.
The main tool there is the theory of de Branges spaces.
Finally, in Section~\ref{section_8},
we relate the theory of the quadratic forms
$Q_W^{\,a}(v)=\langle A_a v,v\rangle_{L^2}$
developed in this paper
to the theory of the quadratic forms
$\langle G_a u,u\rangle_{L^2}$.
While the former belongs to the theory of distributions
and is analytically more delicate,
the latter has the advantage of being treatable within
the theory of integral operators with continuous kernels.

%
\section{Preliminaries} \label{section_2}
%

\subsection{Precise definition of operators} 

Let
\[
L_0^2(-a,a)
=
\{u\in L^2(-a,a)\mid \widehat u(0)=0\},
\]
as in Section~\ref{section_1}.
Then
$D=i\,d/dx$
denotes the differentiation operator on $L^2(-a,a)$
with Dirichlet boundary conditions $v(\pm a)=0$,
whose domain and range are
$\mathfrak D(D)=H_0^1(-a,a)$
and
$\mathfrak R(D)=L_0^2(-a,a)$,
respectively. 
Let $D^\ast$ denote the adjoint of $D$ with respect to 
the standard $L^2$ inner product. 
While $D^\ast$ acts as the same differentiation operator $i\, d/dx$, 
its domain is $\mathfrak{D}(D^\ast) = H^1(-a, a)$ and 
its range is $\mathfrak{R}(D^\ast) = L^2(-a, a)$. 
Consequently, $D$ is not self-adjoint.
Define the projection $P_a : L^2(\R) \to L_0^2(-a,a)$ by 
\begin{equation} \label{EQ_201}
(P_au)(x)
:= 
u(x) - \frac{1}{2a} \int_{-a}^{a} u(y) \, dy
\end{equation}
for $x \in (-a,a)$, and $(P_a u)(x)=0$ otherwise. 
We use the same formula to define $P_a u$
whenever $u|_{(-a,a)}\in L^2(-a,a)$. 
In particular, this gives a meaning to the composition $P_aGP_a$
on $L_0^2(-a,a)$.
Then, 
\[
\aligned 
(P_a G u)(x) 
& = (Gu)(x) - \frac{1}{2a} \int_{-a}^{a} (Gu)(t) \, dt, \quad x \in (-a,a). 
\endaligned 
\]
This defines the operator $G_a$ in \eqref{EQ_105}. 
The kernel function $g$ is locally absolutely continuous and satisfies
$g'\in L^1_{\mathrm{loc}}(\mathbb R)$.
Hence, by Young's inequality on bounded intervals, we have
$G_a u\in\mathfrak{D}(D^\ast)=H^1(-a,a)$ for any $u\in L_0^2(-a,a)$. 
Consequently, the operator $B_a$ on $L^2(-a,a)$ 
with domain 
\[
\mathfrak{D}(B_a):=H_0^1(-a,a)
\]
 is well-defined by \eqref{EQ_106}. 
Since $G_a$ is self-adjoint, the relation 
$\langle B_a v, w \rangle_{L^2}=\langle v, B_aw \rangle_{L^2}$
holds for all $v, w \in \mathfrak{D}(B_a)$. 
Thus, $B_a$ is a symmetric operator, but it is not self-adjoint.
Indeed, the domain of the adjoint operator, $\mathfrak{D}(B_a^\ast)$, 
consists of all $w$ such that the functional $v \mapsto \langle B_a v, w \rangle$ is $L^2$-continuous, which also includes functions $w \in H^1(-a,a)$ that do not vanish on the boundary.
That is, $\mathfrak{D}(B_a) \subsetneq \mathfrak{D}(B_a^\ast)$. 

\subsection{Expansion near the origin} 

We derive an asymptotic formula for the screw function $g$ in a neighborhood of the origin. 
The resulting asymptotic formula will be used repeatedly throughout the paper. 
Applying \cite[Formula 9.554]{GrRy07} with
$z=e^{-2t}$, $m=2$, and $v=1/4$, we obtain
\[
\aligned 
e^{-t/2}\Phi(e^{-2t},2,1/4)
& = 2t(\log(2t)+\psi(1/4)-\psi(2)) \\
& \quad + \zeta(2,1/4) 
+ \sum_{n=2}^{\infty} \zeta(2-n,1/4) \frac{(-2t)^n}{n!}
\endaligned 
\]
for $t \to 0+$, 
where $\zeta(s,a)$ denotes the Hurwitz zeta function. 
By definition \eqref{EQ_103}, 
\[
\aligned 
g(t) 
& =-4(e^{t/2}+e^{-t/2}-2)-(|t|/2)(\psi(1/4)-\log\pi)-(1/4)(F(0)-F(t)) \\
& \quad + \sum_{n \leq \exp(|t|)} \frac{\Lambda(n)}{\sqrt{n}}(|t|-\log n), 
\endaligned 
\]
where 
\[
F(t):=2|t|(\log(2|t|)+\psi(1/4)-\psi(2)) + \zeta(2,1/4)
+ \sum_{n=2}^{\infty} \zeta(2-n,1/4) \frac{(-2|t|)^n}{n!}
\] 
for $|t|$ sufficiently small. 
We have 
$\psi(2)=1-C_0$ \cite[(8.365), (8.366)]{GrRy07} 
and 
$\psi(1/4)=-(\pi/2)-3\log 2-C_0$ \cite[(8.366)]{GrRy07}. 
Motivated by this expansion, we define $r(t)$ for all $t\in\mathbb R$ by
\begin{equation} \label{EQ_202}
\aligned 
g(t) 
& = \frac{1}{2}|t|\log |t|+A|t|
+ \sum_{n \leq \exp(|t|)} \frac{\Lambda(n)}{\sqrt{n}}(|t|-\log n)
+r(t),
\endaligned 
\end{equation}
where
\[
A=\frac{1}{2}(\log(2\pi)-\psi(2))
=\frac{1}{2}(\log(2\pi)+C_0-1)
=0.707546\dots.
\]
The preceding expansion at the origin shows that 
$r$ is $C^2$ in a neighborhood of $t=0$ and satisfies $r(t)=O(t^2)$ as $t\to0$.
On the other hand, \eqref{EQ_103} and \eqref{EQ_202} show that $r$ is smooth away from the origin.
Thus $r$ is an even $C^2$-function on $\mathbb R$.
The sum involving the von Mangoldt function in \eqref{EQ_202} 
is understood to be zero if $|t| <\log 2$.  

\subsection{Quadratic form formulas} For the quadratic form
\[
Q_{B_a}(v) := \langle B_a v, v \rangle_{L^2}, \quad v \in \mathfrak{D}(B_a)=H_0^1(-a,a)
\]
we compute the quadratic forms corresponding to each term on the right-hand side of \eqref{EQ_202}.
For the $|t|$ term on the right-hand side, we have
\[
\int_{-a}^{a}\int_{-a}^{a} |x-y|  v'(y) \overline{v'(x)} dx dy = -2\|v\|_{L^2}^2, 
\quad v \in \mathfrak{D}(B_a).
\]
For the $|t|\log|t|$ term on the right-hand side, if we define
\begin{equation}\label{EQ_203}
\mathcal{L}_a(v)
:=
\frac{1}{4}
\int_{-a}^a
\int_{-a}^a
\frac{|v(x)-v(y)|^2}{|x-y|}
\,dx\,dy
-
\frac{1}{2}
\int_{-a}^a
\log(a^2-x^2)|v(x)|^2\,dx 
\end{equation}
for $v \in \mathfrak{D}(B_a)$ (to ensure the convergence of the second term), 
then we obtain
\begin{equation}\label{EQ_204}
\int_{-a}^a \int_{-a}^a  |x-y|\log|x-y| v'(y) \overline{v'(x)} dx dy 
 = 2 \mathcal{L}_a(v)- 2 \|v\|_{L^2}^2, 
\quad v \in \mathfrak{D}(B_a).
\end{equation}
We verify this relation. Let $I(v)$ denote the left-hand side.
Applying integration by parts with respect to $y$, we have
\[
\int_{-a}^a |x-y|\log|x-y| v'(y) dy 
= - \int_{-a}^a \text{sgn}(y-x)(\log|x-y|+1) v(y) dy.
\]
Substituting this into $I(v)$, we introduce an $\epsilon$-cutoff to handle the singularity of the kernel rigorously:
\[
I(v) = - \lim_{\epsilon \to 0} \int_{-a}^a \overline{v'(x)} \left( \int_{|y-x|>\epsilon} \text{sgn}(y-x)(\log|x-y|+1) v(y) dy \right) dx. 
\]
Let $F_\epsilon(x)$ denote the inner integral inside the brackets. 
For $\operatorname{dist}(x,\partial(-a,a))>\epsilon$,
differentiating it by Leibniz's rule, we obtain 
\begin{align*}
F_\epsilon'(x) 
&= -(\log\epsilon+1){v(x-\epsilon)+ v(x+\epsilon)} - \int_{|y-x|>\epsilon} \frac{v(y)}{|x-y|} dy.
\end{align*}
On the boundary strips
$\operatorname{dist}(x,\partial(-a,a))\leq\epsilon$,
the corresponding formulas contain only the boundary terms that remain
inside $(-a,a)$. Their total contribution tends to zero as
$\epsilon\to0$, since
\[
|v(x)|^2
\leq
\operatorname{dist}(x,\partial(-a,a))
\|v'\|_{L^2}^2.
\]
Using this, we integrate by parts with respect to $x$. 
Taking into account the boundary conditions $v(\pm a)=0$ and the estimate 
$\|v(\cdot\pm\epsilon)-v\|_{L^2}
\leq \epsilon\|v'\|_{L^2}$, 
we find 
\begin{align*}
I(v) 
& = - \lim_{\epsilon \to 0} \int_{-a}^a \left[ 2(\log\epsilon+1)|v(x)|^2
+
\int_{|y-x|>\epsilon} \frac{\overline{v(x)}v(y)}{|x-y|} dy \right] dx.
\end{align*}
Here, we use the identity $2\text{Re}(\overline{v(x)}v(y)) = |v(x)|^2 +  |v(y)|^2- |v(x)-v(y)|^2$ 
to convert the product into a difference of squares 
(the imaginary part vanishes due to the symmetry of the double integral). 
For $\operatorname{dist}(x,\partial(-a,a))>\epsilon$, we have
\[
\int_{y \in (-a,a), |y-x|>\epsilon} \frac{1}{|x-y|} dy
= -2\log\epsilon + \log(a^2-x^2).
\]
The contribution of the boundary strips also tends to zero as
$\epsilon\to0$, by the same estimate above.
Substituting this and simplifying the coefficient of $|v(x)|^2$, 
the contribution from the first term is $2\log\epsilon + 2$, 
while the contribution from the second term is $-2\log\epsilon + \log(a^2-x^2)$. 
Consequently, $\log\epsilon$ is cancelled out, and we obtain
\[
I(v) = - \int_{-a}^a |v(x)|^2 \left( 2 + \log(a^2-x^2) \right) dx 
+ \frac{1}{2} \int_{-a}^a \int_{-a}^a \frac{|v(x)-v(y)|^2}{|x-y|} dx dy.
\]
Rearranging the terms yields \eqref{EQ_204}. 

Furthermore, setting 
$g_0(t)=\sum_{n \leq \exp(|t|)} n^{-1/2}\Lambda(n)\bigl(|t|-\log n\bigr)$
and calculating the corresponding integral, 
\begin{align*}
\,&\int_{-a}^{a}
\int_{-a}^{a}
g_0(x-y)\,
v'(y)\overline{v'(x)}
\,dx\,dy \\
&=
-
\sum_{n\le \exp(2a)}
\frac{\Lambda(n)}{\sqrt{n}}
\Biggl(
\int_{-a}^{\,a-\log n}
v(x+\log n)\overline{v(x)}
\,dx
+
\int_{-a+\log n}^{\,a}
v(x-\log n)\overline{v(x)}
\,dx
\Biggr).
\end{align*}
Computing $Q_{B_a}(v)$ by using the decomposition \eqref{EQ_202}, we get
\begin{equation} \label{EQ_205}
\aligned 
Q_{B_a}(v)
&= \mathcal{L}_a(v) - (2A+1) \|v\|_{L^2}^2 \\
& \quad 
- \sum_{n \leq \exp(2a)} \frac{\Lambda(n)}{\sqrt{n}} 
\Biggl(
\int_{-a}^{\,a-\log n}
v(x+\log n)\overline{v(x)}
\,dx \\
& \qquad \qquad \qquad \qquad \quad 
+
\int_{-a+\log n}^{\,a}
v(x-\log n)\overline{v(x)}
\,dx
\Biggr) \\
& \quad 
- \int_{-a}^{a}\int_{-a}^{a}  r''(x-y) \,v(y) \overline{v(x)} dx dy
\endaligned 
\end{equation}
for $v \in H_0^1(-a,a)$. 
\medskip

\subsection{Formulas in Fourier transforms}~
Expanding the numerator of 
\[
\frac{1}{4}
\iint_{(-a,a)^2,|x-y|>\epsilon}
\frac{|v(x)-v(y)|^2}{|x-y|}
\,dx\,dy 
\]
as
$|v(x)|^2+|v(y)|^2-2\operatorname{Re}(v(x)\overline{v(y)})$
and using symmetry, it equals 
\[
\frac12
\int_{-a}^a
|v(x)|^2
\left(
\int_{(-a,a)\setminus(x-\epsilon,x+\epsilon)}
\frac{1}{|x-y|}
\,dy
\right)
dx
-
\frac{1}{2}
\iint_{(-a,a)^2,\ |x-y|>\epsilon}
\frac{v(x)\overline{v(y)}}{|x-y|}
\,dx\,dy .
\]
For $\operatorname{dist}(x,\partial(-a,a))>\epsilon$, we have
\[
\int_{-a}^{x-\epsilon}\frac{1}{x-y}\,dy
+
\int_{x+\epsilon}^{a}\frac{1}{y-x}\,dy
=
\log(a^2-x^2)-2\log\epsilon.
\]
As in the proof of \eqref{EQ_204}, the contribution of the boundary
strips tends to zero as $\epsilon\to0$. Hence, 
we obtain
\[
\mathcal{L}_a(v)
=
\lim_{\epsilon\to0}
\left[
-\log\epsilon
\int_{-a}^a|v(x)|^2\,dx
-
\frac{1}{2}
\iint_{(-a,a)^2,\ |x-y|>\epsilon}
\frac{v(x)\overline{v(y)}}{|x-y|}
\,dx\,dy
\right].
\]
The second term may be rewritten by Plancherel's theorem,
first truncating the kernel at $|x|=R$ with $R>2a$ and then
passing to the improper Fourier transform of
$
K_\epsilon(x)
=
|x|^{-1}\mathbf 1_{\{|x|>\epsilon\}}(x).
$
For $z\ne0$, this Fourier transform is
\[
\widehat{K_\epsilon}(z)
=
\int_{|x|>\epsilon}
\frac{e^{izx}}{|x|}
\,dx
=
2\int_{\epsilon|z|}^{\infty}
\frac{\cos u}{u}
\,du .
\]
Using the asymptotic expansion of the cosine integral \cite[(8.230)]{GrRy07},
we obtain
\[
\int_{\epsilon|z|}^{\infty}
\frac{\cos u}{u}
\,du
=
-\log\epsilon
-\log|z|
-C_0
+o(1)
\qquad
(\epsilon\to0).
\]
More precisely, writing
\[
m_\epsilon(z)
:=
-\log\epsilon-\frac{1}{2}\widehat{K_\epsilon}(z)
=
-\log\epsilon+\operatorname{Ci}(\epsilon|z|),
\]
we have 
$|m_\epsilon(z)|
\leq C\bigl(1+|\log|z||\bigr)
$ for $0<\epsilon\leq1$ and $z\ne0$. 
Since the zero-extension of $v\in H_0^1(-a,a)$ belongs to
$H^1(\mathbb R)$, its Fourier transform is bounded near the origin
and satisfies $\widehat v(z)=O(|z|^{-1})$ as $|z|\to\infty$.
Hence
$
\bigl(1+|\log|z||\bigr)|\widehat v(z)|^2
\in L^1(\mathbb R),
$
so that dominated convergence applies. 
Substituting this into the Fourier representation yields
\begin{equation} \label{EQ_206}
\aligned
\mathcal{L}_a(v)
&=
\lim_{\epsilon\to0}
\frac{1}{4\pi}
\int_{-\infty}^{\infty}
\bigl(
-2\log\epsilon
-\widehat{K_\epsilon}(z)
\bigr)|\widehat v(z)|^2
\,dz \\
& =
\frac{1}{2\pi}
\int_{-\infty}^{\infty}
\bigl(
\log|z|+C_0
\bigr)|\widehat v(z)|^2
\,dz .
\endaligned
\end{equation}
Using the Fourier transform $\widehat{v}$ of $v$ (defined as the zero-extension of $v$ outside $(-a,a)$) 
\[
\int_{-\infty}^\infty v(x \mp \tau) \overline{v(x)} dx 
= \frac{1}{2\pi} \int_{-\infty}^\infty |\widehat{v}(z)|^2 e^{\pm iz\tau} dz, 
\]
which yields
\[
\aligned 
\int_{-a}^{\,a-\log n}  
v(x+\log n)\overline{v(x)}
\,dx
& +
\int_{-a+\log n}^{\,a}
v(x-\log n)\overline{v(x)}
\,dx \\
& = \frac{1}{\pi} \int_{-\infty}^{\infty} |\widehat{v}(z)|^2 \cos(z\log n) dz.
\endaligned 
\]
Furthermore, noting that 
$\|v\|_{L^2}^2= (2\pi)^{-1}\|\widehat{v}\|_{L^2}^2$, 
we obtain from \eqref{EQ_205} that
\begin{equation} \label{EQ_207}
\aligned 
Q_{B_a}(v)
&= \frac{1}{2\pi} \int_{-\infty}^{\infty} 
\left( \log |z| - \log(2\pi) \right) 
|\widehat{v}(z)|^2 
dz \\
& \quad -
\frac{1}{2\pi} 
 \int_{-\infty}^{\infty} |\hat{v}(z)|^2 
\left[\sum_{n \leq \exp(2a)} \frac{\Lambda(n)}{\sqrt{n}} 2\cos(z\log n) \right] dz \\
& \quad  
- \frac{1}{2\pi} \int_{-\infty}^{\infty} 
\widehat{r_a^{\prime\prime}}(z) |\widehat{v}(z)|^2 dz
\endaligned 
\end{equation}
for $v \in H_0^1(-a,a)$. 
To interpret the term $\widehat{r_a^{\prime\prime}}(z)$,
we decompose $r(t)$ as
\[
r(t)=r_0(t)+r_1(t),
\qquad
r_0(t):=-4(e^{t/2}+e^{-t/2}-2),
\qquad
r_1(t):=r(t)-r_0(t).
\]
By the expansion obtained above, for $|t|$ sufficiently small,
\[
r_1(t)
=
\frac{1}{4}\sum_{n=2}^{\infty}
\zeta(2-n,1/4)\frac{(-2|t|)^n}{n!}.
\]
Since $\widehat{r''}(z)$ cannot be defined directly,
we treat these two parts separately.
For $r_0$, we take the Fourier transform of
$r_{0,a}^{\prime\prime}(x)
:=
r_0^{\prime\prime}(x)\mathbf{1}_{(-2a,2a)}(x)$,
whereas for $r_1$ we use its explicit second derivative.
We then set
\[
\widehat{r_a^{\prime\prime}}(z)
:=
\widehat{r_{0,a}^{\prime\prime}}(z)
+
\widehat{r_1^{\prime\prime}}(z).
\]

Using $\zeta(2-n,a)=-B_{n-1}(a)/(n-1)$ \cite[(25.11.14)]{NIST} 
and the generating series 
$ye^{ay}/(e^y-1)=\sum_{n=0}^{\infty} B_n(a)y^n/n!$ \cite[(9.621)]{GrRy07},
the above local expansion gives
\[
r_1^{\prime\prime}(t)
=
\frac{e^{-|t|/2}}{1-e^{-2|t|}}
-\frac{1}{2|t|}
\]
for $|t|$ sufficiently small and $t\ne0$.
On the other hand, the same expression follows directly from
\eqref{EQ_103}, and hence it holds for every $t\ne0$.
Hence, the Fourier transform is given by
\[
\aligned 
\widehat{r_1^{\prime\prime}}(z) 
& = 2 \int_0^\infty \left( \frac{e^{-t/2}}{1-e^{-2t}} - \frac{1}{2t} \right) \cos(zt) \, dt 
= \int_0^{\infty} \left( \frac{e^{-u/4}}{1-e^{-u}} - \frac{1}{u} \right) \cos(zu/2) \, du \\
& = \lim_{\substack{\epsilon \to +0 \\ R \to \infty}} \left[
\int_\epsilon^R \left( \frac{e^{-u/4} \cos(zu/2)}{1-e^{-u}} - \frac{e^{-u}}{u} \right) du 
+ \int_\epsilon^R \frac{e^{-u} - \cos(zu/2)}{u} \, du \right]
\endaligned 
\]
for real $z\ne0$. 
Differentiating the integral representation
\cite[(8.341.3)]{GrRy07} with respect to its argument, we obtain
\[
\psi(s) = \int_0^\infty
\left(
\frac{e^{-u}}{u}
-
\frac{e^{-su}}{1-e^{-u}}
\right)\,du
\quad (\mathrm{Re}(s)>0).
\]
By taking the real part for $s=1/4+iz/2$, we find
\[
\int_0^\infty \left( \frac{e^{-u/4} \cos(zu/2)}{1-e^{-u}} - \frac{e^{-u}}{u} \right) du
=  -\operatorname{Re}\left[\psi\left(\frac{1}{4} + \frac{iz}{2}\right)\right] 
\]
for real $z\ne0$. On the other hand, we have
\[
\int_\epsilon^\infty \frac{e^{-u}}{u} \, du - \int_\epsilon^\infty \frac{\cos(zu/2)}{u} \, du 
= \mathrm{E}_1(\epsilon) + \mathrm{Ci}(|z/2|\epsilon)
= \log |z| - \log 2+O(\epsilon)
\]
for real $z\ne0$ 
by $\mathrm{E}_1(\epsilon) = -C_0 - \log\epsilon+O(\epsilon)$ \cite[(6.6.2)]{NIST} 
and 
$\mathrm{Ci}(z) = C_0 + \log z+O(|z|)$ \cite[(6.6.6)]{NIST}.
Therefore, 
\[
\widehat{r_1^{\prime\prime}}(z)
=
-\operatorname{Re}\left[\psi\left(\frac{1}{4} + \frac{iz}{2}\right)\right] 
+
\log|z|
- \log 2 = O(|z|^{-2}) 
\quad |z| \to \infty. 
\]
It follows that
\[
\aligned 
Q_{B_a}(v)
&= \frac{1}{2\pi} \int_{-\infty}^{\infty} 
\left( \operatorname{Re}\left[ \psi\left(\frac{1}{4} + \frac{iz}{2}\right) \right]- \log\pi \right) 
|\widehat{v}(z)|^2 
dz \\
& \quad -
\frac{1}{2\pi} 
 \int_{-\infty}^{\infty} |\hat{v}(z)|^2 
\left[\sum_{n \leq \exp(2a)} \frac{\Lambda(n)}{\sqrt{n}} 2\cos(z\log n) \right] dz \\
& \quad 
- \frac{1}{2\pi} \int_{-\infty}^{\infty} 
\widehat{r_{0,a}^{\prime\prime}}(z) |\widehat{v}(z)|^2 dz.
\endaligned 
\]
This formula precisely reflects the structure of the explicit formula. 
From the explicit formula \cite[Lemma 3 and its proof]{Bo01}, we have
\begin{equation*} 
\aligned 
Q_W^{\,a}(v)
& =\frac{1}{2\pi} \int_{-\infty}^{\infty} 
\Re\left[ \psi\left(\frac{1}{4}+\frac{iz}{2}\right) \right]
|\widehat{v}(z)|^2 dz 
-\log\pi \int_{-a}^{a}|v(x)|^2 dx \\ 
& \quad + O_a\bigl(\|v\|_{L^2}^2\bigr).
\endaligned 
\end{equation*}
Since $\log|z|$ is locally integrable and
$|\widehat v(z)|\le \sqrt{2a}\,\|v\|_{L^2}$,
the contribution from a bounded neighborhood of $z=0$
is $O_a(\|v\|_{L^2}^2)$.
Combining this with
\begin{equation*} 
\operatorname{Re}\left[
\psi\left(\frac14+\frac{iz}{2}\right)
\right]
=
\log|z|-\log 2+O(|z|^{-2})
\qquad (|z|\to\infty), 
\end{equation*}
we obtain
\begin{equation} \label{EQ_208}
Q_W^{\,a}(v)=\frac{1}{2\pi} \int_{-\infty}^{\infty} 
\left(\log|z| - \log (2\pi)   \right)
|\widehat{v}(z)|^2 dz 
+  O_a\bigl(\|v\|_{L^2}^2\bigr).
\end{equation}
This perfectly matches the result obtained from \eqref{EQ_207}.

\subsection{Distributional formulas} 

Since $g$ is continuous, it defines a distribution on $\mathbb R$, 
and hence its distributional second derivative $-g''$ is well defined.
By \eqref{EQ_108} and integration by parts,
\begin{equation} \label{EQ_209}
\aligned
Q_W^{\,a}(v)
&=
\int_{-a}^{a}\int_{-a}^{a}
g(x-y)v'(y)\overline{v'(x)}
\,dx\,dy \\
&=
\int_{-a}^{a}\int_{-a}^{a}
(-g''(x-y))v(y)\overline{v(x)}
\,dx\,dy
\endaligned
\end{equation}
for $v\in C_c^\infty(-a,a)$, 
where the second expression is understood in the sense of distributions. 
This was already observed in \cite[Section 3.5]{Su23}. 
Combining this with \eqref{EQ_101}, we formally obtain
\begin{equation} \label{EQ_210}
(A_a v)(x)
=
\int_{-a}^{a}
(-g''(x-y))v(y)\,dy
\end{equation}
for $v\in C_c^\infty(-a,a)$.
The inclusion
$C_c^\infty(-a,a)\subset\mathfrak D(A_a)$,
which justifies this identification rigorously,
will be established in Section~\ref{section_3}. 
Since $g''$ is a distribution and $v\in C_c^\infty(-a,a)$, 
the right-hand side is well defined and smooth as a function of $x$. 
Formula \eqref{EQ_210} provides an alternative representation of Bombieri's Lagrangian 
(\cite[Lemma 1]{Bo01}, \cite[Sections 5--6]{Bo03}). 
Let us compute the right-hand side of \eqref{EQ_210}.
Since
\[
\frac{d^2}{dt^2}(|t|\log|t|)
=
{\rm Pf}\!\left(\frac{1}{|t|}\right)
+
2\delta(t)
\]
in the sense of distributions,
it follows from \eqref{EQ_202} that
\[
\aligned
-g''(t)
&=
-\frac12\,{\rm Pf}\!\left(\frac{1}{|t|}\right)
-(2A+1)\delta(t) \\
&\quad
-\sum_{n\ge2}
\frac{\Lambda(n)}{\sqrt n}
\bigl(\delta(t-\log n)+\delta(t+\log n)\bigr)
-r''(t).
\endaligned
\]
Hence
\begin{equation} \label{EQ_211}
\aligned
\int_{-a}^{a}
(-g''(x-y))v(y)\,dy
&=
-\frac12\,
{\rm Pf}\int_{-a}^{a}
\frac{v(y)}{|x-y|}
\,dy
-(2A+1)v(x) \\
&\quad
-\sum_{n\le e^{2a}}
\frac{\Lambda(n)}{\sqrt n}
\bigl(v(x-\log n)+v(x+\log n)\bigr) \\
&\quad
-\int_{-a}^{a}
r''(x-y)v(y)\,dy .
\endaligned
\end{equation}
Here we note that
\[
\aligned
\int_{-a}^{a}
&
\left(
{\rm Pf}\int_{-a}^{a}
\frac{v(y)}{|x-y|}
\,dy
\right)
\overline{v(x)}
\,dx \\
&=
-\frac12
\int_{-a}^{a}
\int_{-a}^{a}
\frac{|v(x)-v(y)|^2}{|x-y|}
\,dx\,dy
+
\int_{-a}^{a}
\log(a^2-x^2)|v(x)|^2
\,dx 
\endaligned
\]
which follows by a direct symmetrization argument.
Using this identity, one verifies that
\eqref{EQ_209} follows from \eqref{EQ_205}.

For a general element
$v\in\mathfrak{D}(A_a)\subset\mathfrak{D}(Q_W^{\,a})$,
pointwise values need not be well defined, so an expression such as
\eqref{EQ_211} is not meaningful. 
However, since $C_c^\infty(-a,a)$ is dense in
$\mathfrak{D}(Q_W^{\,a})$ with respect to the form norm,
as shown in Section~\ref{section_3},
identity \eqref{EQ_209} extends to
$v\in\mathfrak{D}(A_a)$ by continuity.

%
\section{Proof of Theorem~\ref{thm_1}} \label{section_3}
%

\subsection{The Relation between $A_a$ and $B_a$} 

To compare the operators $A_a$ and $B_a$, we begin by recalling some
basic facts concerning the quadratic form $Q_W^{\,a}$. 
For general background on quadratic forms, we refer to Schm\"udgen~\cite[Chapter 10]{Sch12}.

\medskip

Let $\Gamma$ denote the set of all zeros of the function
$z \mapsto \xi(1/2-iz)$.
For each $\gamma \in \Gamma$, let $m_\gamma$ be the multiplicity of $\gamma$
as a zero of $\xi(1/2-iz)$.
By the functional equations
$\xi(1-s)=\xi(s)$ and $\xi(s)=\overline{\xi(\bar{s})}$,
the set $\Gamma$ is invariant under
$\gamma \mapsto -\gamma$ and $\gamma \mapsto \bar{\gamma}$.
It is customary to write the zeros of $\xi(s)$ as
$\rho=\beta+i\gamma$ ($\beta,\gamma\in\mathbb R$),
but our notation differs from this convention.
In particular, an element $\gamma\in\Gamma$ may be non-real.
RH is equivalent to the inclusion
$\Gamma\subset\mathbb R$.
In this notation, the Weil explicit formula yields 
\begin{equation} \label{EQ_301}
Q_W(v_1,v_2)
=
\sum_{\gamma\in\Gamma}
m_\gamma
\widehat{v_1}(\gamma)
\overline{\widehat{v_2}(\bar{\gamma})}
\end{equation}
(cf. \cite[(1.2), (3.3)]{Su25}).
The form domain of the closed symmetric quadratic form $Q_W^{\,a}$ is defined by
\[
\mathfrak{D}(Q_W^{\,a})
:=
\{\,v\in L^2(-a,a)\mid |Q_W^{\,a}(v)|<\infty\,\},
\]
and becomes a Hilbert space with respect to the form norm
\[
\|v\|_{Q_W^{\,a}}^2
:=
Q_W^{\,a}(v)
+
(1-\lambda_a)\|v\|_{L^2}^2,
\]
where $\lambda_a$ is the constant defined in \eqref{EQ_107}.
Define
\[
\aligned
H^{\log}(-a,a)
:&=
\left\{
v\in L^2(-a,a)\ \middle|\
\int_{-\infty}^{\infty}
(1+\log^+|z|)
|\widehat v(z)|^2\,dz
<\infty
\right\}
\\
&
=
\{\,v\in L^2(-a,a)\mid \mathcal L_a(v)<\infty\,\}, 
\endaligned
\]
where $\mathcal L_a$ is understood as the extended quadratic form
defined by the right-hand side of \eqref{EQ_203}.
Since $v$ is supported in $[-a,a]$, its Fourier transform
$\widehat v$ is bounded, and hence the contribution of
$\log|z|$ near $z=0$ is finite.
Thus \eqref{EQ_206} shows that the two finiteness conditions in the definition above are equivalent.
It follows from \eqref{EQ_208} that
\[
\mathfrak{D}(Q_W^{\,a})
=
H^{\log}(-a,a).
\]

In \cite[Theorem 5.1]{Bo03}, Bombieri implicitly works with the space of
all $v\in L^2(-a,a)$ such that
\[
\int_{-a}^{a}\int_{-a}^{a}
\frac{|v(x)-v(y)|^2}{|x-y|}
\,dx\,dy
+
\int_{-a}^{a}|v(x)|^2\,dx
<\infty.
\]
On a bounded interval, the logarithmic endpoint term in
\eqref{EQ_203} is controlled by the double-integral term together with
the $L^2$ norm. Hence, in view of \eqref{EQ_206} and \eqref{EQ_208},
this space coincides with $\mathfrak D(Q_W^{\,a})$.

\medskip

The fact that $Q_W^{\,a}(v)$ is lower bounded and lower semicontinuous on
$L^2(-a,a)$ was proved in \cite[Proposition 2.1]{CC23}.
It is also mentioned in the proof of \cite[Lemma 3]{Bo01} and is implicitly contained in \cite[Section 2]{Yo92}.
All of these arguments are based on the formula \eqref{EQ_208}.
The lower semicontinuity follows immediately from an application of Fatou's lemma.

We now recall the self-adjoint operator $A_a$ associated with
$Q_W^{\,a}$ through \eqref{EQ_101}.
Such a self-adjoint operator is uniquely determined.
Moreover,
$v\in\mathfrak{D}(A_a)\subset\mathfrak{D}(Q_W^{\,a})$
if and only if the map
$
w
\mapsto
Q_W^{\,a}(v,w)
$
is continuous on $\mathfrak{D}(Q_W^{\,a})$ with respect to the $L^2$-norm. 
By the Riesz representation theorem, this is exactly the condition
in \cite[Definition 10.4]{Sch12}.
In this case, $A_av$ is uniquely characterized by
$
Q_W^{\,a}(v,w)
=
\langle A_av,w\rangle_{L^2}
$ ($w\in\mathfrak{D}(Q_W^{\,a})$). 

\begin{lemma} \label{lem_3_1}
$B_a \subset A_a$.
In particular, $A_a$ is a self-adjoint extension of $B_a$.
\end{lemma}
\begin{proof}
Let $E \subset L^2(-a,a)$ be the linear subspace generated by
\[
e_n(x)=\exp(i\pi n x/a),
\qquad n\in\mathbb Z.
\]
The space $E$ is a core for the quadratic form $Q_W^{\,a}$
\cite[Proposition 2.3]{CC23}.
On the other hand, if $v\in \mathfrak{D}(B_a)=H_0^1(-a,a)$, then
$\widehat v(z)=O(|z|^{-1})$ as $|z|\to\infty$.
Hence, by \eqref{EQ_301},
$
\mathfrak D(B_a)
\subset
\mathfrak D(Q_W^{\,a}). 
$ 
Moreover,
\[
\widehat{e_n}(z)
=
\frac{2\sin(az+n\pi)}{z+n\pi/a},
\]
so that $\widehat w(z)=O(|z|^{-1})$ for every $w\in E$.
If $u=Dv$, then
$\widehat v(z)=z^{-1}\widehat u(z)$.
Therefore, by \eqref{EQ_301} and \cite[(3.1)]{Su23},
$
Q_W^{\,a}(v,w)
=
\langle B_a v,w\rangle_{L^2}
$ 
for all
$v\in H_0^1(-a,a)=\mathfrak D(B_a)$
and
$w\in E$. 
It follows from \cite[Proposition 10.5\,(v)]{Sch12}
that $B_a\subset A_a$.
\end{proof}

\subsection{Completion of the proof}

For simplicity, we write
$A=A_a$, $B=B_a$, $Q_B(v)=\langle Bv, v \rangle_{L^2}$,
$Q=Q_W^{\,a}$, and
$\|v\|_{Q}^2 = Q(v) + (1-\lambda_a)\|v\|_{L^2}^2$. 
Note that $Q=Q_B$ on $\mathfrak{D}(B_a)$. 
The Friedrichs extension of the operator $B$
is the uniquely determined self-adjoint extension associated with
the closure $\overline{Q}_B$ of the quadratic form $Q_B$.
The domain $\mathfrak{D}(\overline{Q}_{B})$ of $\overline{Q}_B$
is defined as the completion of the form domain
$\mathfrak{D}(Q_{B}):=\mathfrak{D}(B)=H_0^1(-a,a)$
with respect to the form norm $\|\cdot\|_Q$.
Since $B \subset A$ (Lemma~\ref{lem_3_1}),
we have
$\mathfrak{D}(\overline{Q}_{B}) \subset \mathfrak{D}(Q_{A})=\mathfrak{D}(Q)$.
Hence, $\mathfrak{D}(\overline{Q}_{B})$ is a closed subspace of
$\mathfrak{D}(Q)$ with respect to the form norm.
On the other hand, the space $E$ 
in the proof of Lemma~\ref{lem_3_1} 
is a core for the quadratic form $Q$
(i.e. $\overline{E}^{\,\|\cdot\|_Q} = \mathfrak{D}(Q)$).
Since $C_c^\infty(-a,a) \subset H_0^1(-a,a)$,
if
\begin{equation} \label{EQ_302}
E \subset \overline{C_c^\infty(-a,a)}^{\,\|\cdot\|_Q},
\end{equation}
then
$\mathfrak{D}(\overline{Q}_{B})=\mathfrak{D}(Q)$.
It follows that the Friedrichs extension of $B$ coincides with $A$.
Indeed, by \cite[Corollary 10.8]{Sch12},
on a given Hilbert space, lower bounded self-adjoint operators
are in one-to-one correspondence with lower bounded closed symmetric
quadratic forms having dense domains.
\medskip

To prove the inclusion \eqref{EQ_302}, it suffices to show that each function
$e_n\in E$ belongs to the closure of $C_c^\infty(-a,a)$
with respect to the form norm $\|\cdot\|_Q$.
In other words, we construct
$v_\epsilon\in C_c^\infty(-a,a)$ such that 
$
\|e_n-v_\epsilon\|_Q\longrightarrow 0$ 
as $\epsilon\to 0$. 
For sufficiently small $\epsilon>0$, choose
$\eta_\epsilon\in C_c^\infty(-a,a)$ such that 
$0\leq\eta_\epsilon\leq 1$, 
$\eta_\epsilon(x)=1$ for $x\in[-a+2\epsilon,a-2\epsilon]$, 
and 
$
\eta_\epsilon(x)=0
$ for $
x\in[-a,-a+\epsilon]\cup[a-\epsilon,a].
$
We may also assume that 
$|\eta_\epsilon'(x)|\ll\epsilon^{-1}.$

Define
\[
v_\epsilon(x):=e_n(x)\eta_\epsilon(x),
\qquad
w_\epsilon(x):=e_n(x)-v_\epsilon(x)
=e_n(x)(1-\eta_\epsilon(x)).
\]
Then $v_\epsilon\in C_c^\infty(-a,a)$, while $w_\epsilon$ is supported in 
$[-a,-a+2\epsilon]\cup[a-2\epsilon,a]$. 
Since $|w_\epsilon(x)|\leq1$, we have
$\|w_\epsilon\|_{L^1}\leq4\epsilon$, 
$\|w_\epsilon\|_{L^2}^2\leq4\epsilon$. 
From the Fourier representation of \(Q_W^{\,a}\) obtained in Section~\ref{section_2}, 
there exists a constant \(C_a>0\), independent of $\epsilon$, such that
\[
\|w_\epsilon\|_Q^2
\leq
C_a
\int_{-\infty}^{\infty}
(1+\log^+|z|)
|\widehat{w_\epsilon}(z)|^2\,dz.
\]
Since $
1+\log^+|z|
\ll
\log(3+|z|),
$ 
it is enough to show that
\[
I_\epsilon
:=
\int_{-\infty}^{\infty}
\log(3+|z|)
|\widehat{w_\epsilon}(z)|^2\,dz
~\to~0.
\]

First, we have
\begin{equation} \label{EQ_303}
|\widehat{w_\epsilon}(z)|
\leq
\|w_\epsilon\|_{L^1}
\leq
4\epsilon.
\end{equation}
On the other hand,
$|\widehat{w_\epsilon}(z)|=O(|z|^{-1})$ as $|z|\to\infty$,
uniformly in $\epsilon$.
Indeed, integration by parts gives
\[
\widehat{w_\epsilon}(z)
=
\left[
w_\epsilon(x)\frac{e^{izx}}{iz}
\right]_{-a}^{a}
-
\frac{1}{iz}
\int_{-a}^{a}
w_\epsilon'(x)e^{izx}\,dx.
\]
Moreover,
\[
w_\epsilon'(x)
=
e_n'(x)(1-\eta_\epsilon(x))
-
e_n(x)\eta_\epsilon'(x),
\]
and hence
\[
\|w_\epsilon'\|_{L^1}
\ll
\epsilon
+
\int_{-a}^{a}|\eta_\epsilon'(x)|\,dx
\ll1.
\]
Since $
|w_\epsilon(a)|=|w_\epsilon(-a)|=1$,
there exists a constant $M>0$, independent of $\epsilon$, such that
\begin{equation} \label{EQ_304}
|\widehat{w_\epsilon}(z)|
\leq
\frac{M}{|z|}.
\end{equation}

We split $I_\epsilon$ at $|z|=1/\epsilon$.
By \eqref{EQ_303},
\[
\begin{aligned}
\int_{|z|\leq1/\epsilon}
\log(3+|z|)
|\widehat{w_\epsilon}(z)|^2\,dz
&\ll
\epsilon^2
\int_{|z|\leq1/\epsilon}
\log(3+|z|)\,dz
\ll
\epsilon\log(1/\epsilon).
\end{aligned}
\]
By \eqref{EQ_304},
\[
\begin{aligned}
\int_{|z|>1/\epsilon}
\log(3+|z|)
|\widehat{w_\epsilon}(z)|^2\,dz
&\ll
\int_{1/\epsilon}^{\infty}
\frac{\log(3+z)}{z^2}\,dz
\ll
\epsilon\log(1/\epsilon).
\end{aligned}
\]
Therefore,
$
I_\epsilon
\ll
\epsilon\log(1/\epsilon)
\to 0,
$
and consequently
$
\|e_n-v_\epsilon\|_Q
=
\|w_\epsilon\|_Q
\to 0.
$
Thus every $e_n\in E$ belongs to
$
\overline{C_c^\infty(-a,a)}^{\,\|\cdot\|_Q},
$ 
and hence \eqref{EQ_302} follows.
Therefore,
$
\mathfrak D(\overline Q_B)=\mathfrak D(Q),
$
so the Friedrichs extension of $B_a$ coincides with $A_a$.
\hfill $\Box$

%
\section{Proof of Theorem~\ref{thm_3}} \label{section_4}
%

\subsection{Compact embedding result} 

The following proposition is the main analytic ingredient in the proof of Theorem~\ref{thm_3}.

\begin{proposition} \label{prop_260309}
For a bounded interval $I$, the embedding
$
H^{\log}(I)\hookrightarrow L^2(I)
$ 
is compact.
In other words, 
every bounded sequence in $H^{\log}(I)$ admits a subsequence converging in $L^2(I)$.
\end{proposition}

\begin{proof}
The proof is essentially the same as that of \cite[Theorem 3.6]{CCM25}.
\end{proof}

This compact embedding is also the key ingredient 
in the proof of the discreteness of the spectrum of
$A_a$ in \cite[Theorem 3.6]{CCM25}. 
By \cite[Proposition 10.6]{Sch12}, 
it suffices to prove that the embedding
$ 
(\mathfrak{D}(Q_W^{\,a}),\|\cdot\|_{Q_W^{\,a}})
\hookrightarrow
(L^2(-a,a),\|\cdot\|_{L^2})
$ 
is compact.
This follows readily from
Proposition~\ref{prop_260309}. 

\subsection{Scaling transformation of the Rayleigh quotient}

Using \eqref{EQ_108}, 
we transfer the Rayleigh quotient in \eqref{EQ_107} 
to the fixed interval $[-1,1]$ by a scaling transformation.
From definition \eqref{EQ_106}, 
the Rayleigh quotient associated with $B_a$ for $v \in H_0^1(-a, a)$ is given by
\[
\frac{\langle B_a v, v \rangle_{L^2}}{\|v\|_{L^2}^2} 
=
\frac{\langle G_a Dv, Dv \rangle_{L^2}}{\|v\|_{L^2}^2} 
=
\frac{\displaystyle \int_{-a}^a \int_{-a}^a g(x-y) v'(y) \overline{v'(x)} \, dx dy}{\displaystyle \int_{-a}^a |v(x)|^2 dx}.
\]
Setting $w(t)=v(at)$ on the fixed interval $[-1,1]$, we define
\begin{equation} \label{EQ_401}
q_a(w):=
 \int_{-1}^1 \int_{-1}^1 \frac{1}{a} g(a(x-y))
 w'(x)\overline{w'(y)} \,dx\,dy .
\end{equation}
Then a change of variables gives
\begin{equation} \label{EQ_402}
\frac{\langle B_a v, v \rangle_{L^2}}{\|v\|_{L^2}^2}
=
\frac{q_a(w)}{\|w\|_{L^2}^2},
\qquad
w(t)=v(at)\in H_0^1(-1,1).
\end{equation}
We write
\begin{equation} \label{EQ_403}
R(a,w):=\frac{q_a(w)}{\|w\|_{L^2}^2},\qquad (0\ne w\in H_0^1(-1,1)).
\end{equation}

For the second term on the right-hand side of \eqref{EQ_203}, we have
\[
\int_{-a}^{a} |v(x)|^2 \log(a^2-x^2) \,dx
=
2\|v\|_{L^2}^2\log a
+
a\int_{-1}^{1} |v(ax)|^2 \log(1-x^2) \,dx.
\]
Thus, defining
\begin{equation} \label{EQ_404} 
\begin{aligned}
\mathcal{L}(w)
:=
\frac{1}{4}
\int_{-1}^1
\int_{-1}^1
\frac{|w(x)-w(y)|^2}{|x-y|}
\,dx\,dy
-
\frac{1}{2}
\int_{-1}^1
|w(x)|^2\log(1-x^2)\,dx 
\end{aligned}
\end{equation}
for $w \in H_0^1(-1,1)$,  
we deduce from \eqref{EQ_205} that
\begin{equation} \label{EQ_405}
\aligned 
R(a,w)
& = - \log a -(2A+1) + \frac{\mathcal{L}(w)}{\|w\|_{L^2}^2} \\
& \quad-\frac{1}{\|w\|_{L^2}^{2}}
\sum_{n\le e^{2a}}
\frac{\Lambda(n)}{\sqrt{n}}
\Biggl(
\int_{-1}^{\,1-\frac{\log n}{a}}
w\left(x+\frac{\log n}{a}\right)
\overline{w(x)}
\,dx
\\
&\qquad\qquad\qquad\qquad\qquad\qquad
+
\int_{-1+\frac{\log n}{a}}^{\,1}
w\left(x-\frac{\log n}{a}\right)
\overline{w(x)}
\,dx
\Biggr) \\
& \quad - \frac{a}{\|w\|_{L^2}^2} 
 \int_{-1}^1 \int_{-1}^1 r''(a(x-y)) w(y) \overline{w(x)} \, dx dy
\endaligned 
\end{equation}
for $w \in H_0^1(-1,1)$.

For the quadratic form $\mathcal{L}$ defined by \eqref{EQ_404} 
for $w \in \mathfrak{D}(\mathcal{L}):=H_0^1(-1,1)$, 
one can show in the same way as in the proof of \eqref{EQ_206} that
\begin{equation} \label{EQ_406}
\mathcal L(w)
=
\frac{1}{2\pi}
\int_{-\infty}^{\infty}
\bigl(
\log|z|+C_0
\bigr)|\widehat w(z)|^2
\,dz .
\end{equation}
Hence the form norm $\|w\|_{\mathcal L}^2:=\mathcal L(w)+\|w\|_{L^2}^2$ 
is equivalent to $\int_{-\infty}^{\infty}(1+\log^+|z|)|\widehat{w}(z)|^2 \,dz$. 
Therefore, the completion of $\mathfrak{D}(\mathcal{L})$ with respect to $\|\cdot\|_{\mathcal L}$ 
can be identified with a subspace of $H^{\log}(-1,1)$.
In particular, formula \eqref{EQ_406} is valid not only 
for $w \in \mathfrak{D}(\mathcal{L})$, but also for $w \in \mathfrak{D}(\overline{\mathcal{L}})$.
\medskip

To prove Theorem~\ref{thm_3}, 
it remains to establish the upper semicontinuity
\begin{equation} \label{EQ_407}
\limsup_{a \to a_0} \lambda_a 
\leq \lambda_{a_0}. 
\end{equation}
and lower semicontinuity
\begin{equation} \label{EQ_408}
\liminf_{n \to \infty} \lambda_{a_n} \ge \lambda_{a_0}. 
\end{equation}
From \eqref{EQ_402} and Theorem~\ref{thm_1}, 
we obtain for the Rayleigh quotient associated with the closure $\bar{q}_a$ of $q_a$ that
\begin{equation} \label{EQ_409}
\frac{Q_W^{a}(v)}{\|v\|_{L^2}^2} 
=\frac{\bar{q}_a(w)}{\|w\|_{L^2}^2},  \quad v(at)=w(t) \in \mathfrak{D}(\bar{q}_a).
\end{equation}
By \eqref{EQ_405}, \eqref{EQ_406}, and the definition of $H^{\log}(-1,1)$, 
we obtain 
$\mathfrak{D}(\bar{q}_a) \subset H^{\log}(-1,1)$ for every $a>0$. 
Hence, by \eqref{EQ_107} and \eqref{EQ_409}, 
it suffices to study the infimum of
\[
\bar{R}(a,w) := \frac{\bar{q}_a(w)}{\|w\|_{L^2}^2}.
\]

\subsection{Proof of upper semicontinuity}

Let $q_a$ be the quadratic form defined in \eqref{EQ_401}, 
and let $R(a,w)$ be the Rayleigh quotient defined in \eqref{EQ_403}.
We choose $w_0 \in \mathfrak{D}(\bar{q}_{a_0}) \subset H^{\log}(-1,1)$ 
such that $\lambda_{a_0} = \bar{R}(a_0, w_0)$, 
and normalize it by $\|w_0\|_{L^2}=1$.
To avoid referring to pointwise values of $w_0$ and $\widehat{w_0}$, 
we choose a sequence $w_n \in C_c^\infty(-1,1)$ 
such that
$w_n \to w_0$ 
with respect to the form norm $\|\cdot\|_0$ 
(the form norm associated with $q_{a_0}$), and $\|w_n\|_{L^2}=1$.
Such a sequence exists by the proof of Theorem~\ref{thm_1}.
Since $C_c^\infty(-1,1) \subset \mathfrak{D}(q_a)$, 
each $w_n$ belongs to $\mathfrak{D}(q_a)$ for every $a$, 
and $R(a,w_n)$ can be expressed in terms of $w_n$ and $\widehat{w_n}$.

Fix $n$. 
By the definition of the infimum 
$\lambda_a \leq R(a,w_n)$ for every $a$.
Moreover, 
\eqref{EQ_405} implies that 
$R(a,w_n) \to R(a_0,w_n)$ as $a \to a_0$. 
Since $\|w_n\|_{L^2}=1$, we obtain
\[
\limsup_{a \to a_0} \lambda_a \leq R(a_0, w_n) = q_{a_0}(w_n).
\]
Passing to the limit as $n\to\infty$, we obtain
\[
\limsup_{a \to a_0} \lambda_a 
\leq 
\lim_{n \to \infty} q_{a_0}(w_n)
= \bar{q}_{a_0}(w_0)
= \bar{R}(a_0,w_0)
= \lambda_{a_0}.
\]
Thus, \eqref{EQ_407} holds. \hfill $\Box$

\subsection{Proof of lower semicontinuity}

By \eqref{EQ_405}, we decompose 
$\bar{q}_a=\bar{q}^0+\bar{q}_a^1$, 
where
$
\bar{q}^0(w)
=
\overline{\mathcal L}(w)
-(2A+1)\|w\|_{L^2}^2
$
and $\bar{q}_a^1$ denotes the remaining terms in
\eqref{EQ_405}.
Let $(a_n)$ be a sequence with $a_n \to a_0$, 
and let $w_n \in \mathfrak{D}(\bar{q}_{a_n}) \subset H^{\log}(-1,1)$ 
be a sequence such that $\lambda_{a_n} = \bar{R}(a_n, w_n)$ and $\|w_n\|_{L^2}=1$. 
Passing to a subsequence if necessary, we may assume that 
$\lambda_{a_n}\to \liminf_{m\to\infty}\lambda_{a_m}$.
We first show that there exists a subsequence of $(w_n)$ 
which converges in $L^2$ to some limit $w_\ast$. 
Fix an arbitrary nonzero
$w \in H_0^1(-1,1) \subset H^{\log}(-1,1)$.
Then $\lambda_{a_n} \leq R(a_n,w)$, 
and since $R(a,w)$ is continuous in $a$ by the formula \eqref{EQ_405}, 
the sequence $(\lambda_{a_n})$ is  bounded above.

Since 
$\lambda_{a_n} = \bar{q}_{a_n}(w_n)$ by $\lambda_{a_n} = \bar{R}(a_n,w_n)$ and $\|w_n\|_{L^2}=1$,
the sequence
$\bar q_{a_n}(w_n)$
is bounded above. 
As $(a_n)$ remains in a compact neighborhood of $a_0$,
the estimate
$\bar q_a^1(v)=O(\|v\|_{L^2}^2)$
holds uniformly in $a$.
It follows that the terms other than
$\overline{\mathcal L}(w_n)$
in \eqref{EQ_405}
are uniformly bounded.
Consequently,
$\overline{\mathcal L}(w_n)$ is bounded above.
Since $\overline{\mathcal L}$ is bounded below by definition,
the sequence
$\overline{\mathcal L}(w_n)$
is bounded.
Together with \eqref{EQ_406},
this shows that $(w_n)$ is bounded in $H^{\log}(-1,1)$.
Proposition~\ref{prop_260309} therefore implies that
$(w_n)$ is relatively compact in $L^2$,
and we may extract a subsequence converging in $L^2$. 
We relabel it again as $(w_n)$ and denote its limit by $w_\ast$.

The explicit expression for $\bar q_a^1$ shows that it is a finite sum
of quadratic forms associated with translation operators and integral
operators with continuous kernels. 
Thus $\bar q_a^1(w)$ depends continuously on $(a,w)$ near $(a_0,w_\ast)$, 
and therefore
\[
\lim_{n\to\infty}\bar q_{a_n}^1(w_n)
=
\bar q_{a_0}^1(w_\ast).
\]
On the other hand, $\bar{q}^0$ is an $a$-independent 
lower bounded closed quadratic form on $L^2(-1,1)$. 
Hence, by \cite[Proposition 10.1]{Sch12}, it is lower semicontinuous with respect to $L^2$-convergence:
\[
\liminf_{n \to \infty} \bar{q}^0(w_n) \ge \bar{q}^0(w_\ast).
\]
Combining these estimates, we obtain
\[
\liminf_{n \to \infty} \lambda_{a_n}
= \liminf_{n \to \infty} \left( \bar{q}^0(w_n) + \bar{q}_{a_n}^1(w_n) \right)
\ge \bar{q}^0(w_\ast) + \bar{q}_{a_0}^1(w_\ast)
= \bar{q}_{a_0}(w_\ast).
\]
Since $w_n \to w_\ast$ in $L^2(-1,1)$ and
$\|w_n\|_{L^2}=1$ for all $n$,
we have $\|w_\ast\|_{L^2}=1$. 
Therefore,
$
\bar R(a_0,w_\ast)
=
\bar q_{a_0}(w_\ast),
$
and by the definition of $\lambda_{a_0}$ as the infimum of
$\bar R(a_0,\cdot)$,
$
\bar q_{a_0}(w_\ast)
=
\bar R(a_0,w_\ast)
\ge \lambda_{a_0}.
$
Consequently,
\[
\liminf_{n \to \infty} \lambda_{a_n}
\ge
\bar q_{a_0}(w_\ast)
\ge
\lambda_{a_0},
\]
which proves \eqref{EQ_408}.
\hfill $\Box$

\subsection{A remark on the continuity of $\lambda_a$ via parity decomposition} 
\label{section_0602}

Since the kernel $g(x-y)$ is even, 
$G_a$ commutes with the parity operator $(Ju)(x)=u(-x)$, i.e. $G_a J = J G_a$. 
It follows that both
$B_a=D^\ast G_aD$
and its Friedrichs extension $A_a$
commute with $J$.
In particular, each eigenspace $E(\lambda)$ of $A_a$ admits a decomposition into even and odd subspaces:
$E(\lambda)=E^+(\lambda)\oplus E^-(\lambda)$.
For the lowest eigenvalue space
$E(\lambda_a)=E^+(\lambda_a)\oplus E^-(\lambda_a)$, 
the assumption in \cite[Theorem 5.10]{CCM25}
amounts to requiring that
${\rm dim}\,E^+(\lambda_a)=1$ and $E^-(\lambda_a)=\{0\}$. 

In view of the parity decomposition and Corollary~\ref{cor_2}, 
we define
\[
\lambda_a^\bullet :=
\inf_{0 \ne v \in H_\bullet^1(-a,a)}
\frac{\langle B_a v, v \rangle_{L^2}}{\|v\|_{L^2}^2},
\qquad
\bullet \in \{+,-\}, 
\]
where $H_+^1(-a,a)$ and $H_-^1(-a,a)$ denote the even and odd
subspaces of $H_0^1(-a,a)$, respectively. 
Then the global infimum $\lambda_a$ in \eqref{EQ_107} can be expressed as
\begin{equation} \label{EQ_410}
\lambda_a = \min(\lambda_a^+, \lambda_a^-).
\end{equation}
Thus the continuity of $\lambda_a$
follows from that of $\lambda_a^\pm$,
since the minimum of two continuous functions is continuous.
The continuity of $\lambda_a^\pm$
is asserted in \cite[Theorem 5]{Bo01},
although the details of the proof are not fully provided there.

We now verify that \eqref{EQ_410} holds. 
Every $v \in H_0^1(-a,a)$ admits a unique decomposition into even and odd components: 
$v = v^+ + v^-$. Since differentiation reverses parity and $G_a$ preserves it,
$Dv^+$ and $G_a(Dv^+)$ are odd,
whereas $Dv^-$ and $G_a(Dv^-)$ are even.
Expanding the quadratic form in the numerator, 
\[
\langle G_aDv, Dv \rangle
=
\langle G_a(Dv^+), Dv^+ \rangle
+
\langle G_a(Dv^-), Dv^- \rangle, 
\]
since the cross terms vanish due to the orthogonality. 
Similarly,
$
\|v\|_{L^2}^2 = \|v^+\|_{L^2}^2 + \|v^-\|_{L^2}^2 
$ for the denominator. 
Hence the Rayleigh quotient on $H_0^1(-a,a)$ can be written as
\[
\frac{\langle G_aDv, Dv \rangle}{\|v\|_{L^2}^2}
=
\frac{
\langle G_a(Dv^+), Dv^+ \rangle
+
\langle G_a(Dv^-), Dv^- \rangle
}{
\|v^+\|^2 + \|v^-\|^2
}.
\]
For each component we have
$
\langle G_a(Dv^\bullet), Dv^\bullet \rangle \geq \lambda_a^\bullet \|v^\bullet\|^2
$ 
for 
$\bullet \in \{+,-\}$ by definition.
Therefore,
\[
\frac{\langle G_aDv, Dv \rangle}{\|v\|_{L^2}^2}
\geq
\frac{\lambda_a^+ \|v^+\|^2 + \lambda_a^- \|v^-\|^2}{\|v^+\|^2 + \|v^-\|^2}
\geq \min(\lambda_a^+, \lambda_a^-),
\]
which yields $\lambda_a \geq \min(\lambda_a^+, \lambda_a^-)$.
Conversely,
restricting the Rayleigh quotient to
$H_+^1(-a,a)$
and
$H_-^1(-a,a)$
immediately yields
$\lambda_a \leq \min(\lambda_a^+, \lambda_a^-)$. 
Thus, \eqref{EQ_410} is established.

%
\section{Proof of Theorem~\ref{thm_4}} \label{section_5}
%

\subsection{Positivity of the lower bound}

By \eqref{EQ_409}, the behavior of $\lambda_a$ for small $a>0$
is determined by the asymptotics of the Rayleigh quotient
$R(a,w)$ in \eqref{EQ_403} as $a \to 0+$. 
From \eqref{EQ_405}, we have
\begin{equation} \label{EQ_501}
R(a,v)  = \log \frac{1}{a} - (2A + 1) + \frac{\mathcal{L}(v)}{\|v\|^2} + O(a)
\end{equation}
for $0<a<(1/2)\log 2$ and $v \in H_0^1(-1,1)$, so it suffices to study $\mathcal{L}(v)/\|v\|^2$.
Note that this quantity is independent of $a$. 
In general, the infimum of $\mathcal{L}(v)/\|v\|^2$ need not coincide with that of $R(a,v)$, 
and hence their minimizers need not agree, but this distinction will not be relevant in what follows.
By \eqref{EQ_406}, the domain $\mathfrak{D}(\bar{\mathcal{L}})$ of the closure $\bar{\mathcal{L}}$
is contained in $H^{\log}(-1,1)$. 
Furthermore, 
the self-adjoint operator $T$ associated with 
$\bar{\mathcal{L}}$ has discrete spectrum. 
This can be proved by the same argument as in the proof of
\cite[Theorem 3.6]{CCM25}, using Proposition~\ref{prop_260309}.

The term $\log(1/a)$ on the right-hand side of \eqref{EQ_501} diverges to $+\infty$ as $a \to 0+$. 
Since the closure $\bar{\mathcal{L}}$ is a lower bounded closed quadratic form,
it is lower semicontinuous \cite[Proposition 10.1]{Sch12}.
We claim that
$
\inf_{\|v\|=1}\bar{\mathcal{L}}(v)>0.
$
Indeed, it follows directly from the definition that $\mathcal{L}(v)\ge0$. 
If there existed a sequence $(v_n)$ with
$\|v_n\|_{L^2}=1$ such that $\mathcal{L}(v_n) \to 0$, then
the first term in \eqref{EQ_404} would also tend to $0$, 
hence a subsequence of $(v_n)$ converges in $L^2(-1,1)$ to a constant function 
by \eqref{EQ_406} and Proposition~\ref{prop_260309}. 
However, the second term in \eqref{EQ_404} is strictly positive for constant functions.
Therefore, for sufficiently small $a>0$, we obtain $\lambda_a = \inf_w R(a,w) > 0$.

\subsection{Positivity improving property}

The positivity improving property of the semigroup associated with
$\bar{\mathcal{L}}$ established in this subsection
will play a key role in proving the simplicity of the lowest eigenvalue.
To this end, we first show that
$(\bar{\mathcal L},\mathfrak D(\bar{\mathcal L}))$
is a Dirichlet form.
Let
$
\phi(t):=\min(1,\max(0,t))
$
and, for a real-valued $v\in\mathfrak D(\mathcal L)$, put
$
v^\sharp:=\phi\circ v=(0\vee v)\wedge1.
$
Since $\phi$ is Lipschitz continuous and $\phi(0)=0$, we have
$
v^\sharp\in H_0^1(-1,1)=\mathfrak D(\mathcal L).
$
For the first term in \eqref{EQ_404}, the Lipschitz continuity of $\phi$ yields
\[
|v^\sharp(x)-v^\sharp(y)|^2
=
|\phi(v(x))-\phi(v(y))|^2
\le
|v(x)-v(y)|^2.
\]
Hence
\[
\frac14
\int_{-1}^1\int_{-1}^1
\frac{|v^\sharp(x)-v^\sharp(y)|^2}{|x-y|}
\,dx\,dy
\le
\frac14
\int_{-1}^1\int_{-1}^1
\frac{|v(x)-v(y)|^2}{|x-y|}
\,dx\,dy.
\]
For the second term in \eqref{EQ_404}, we have
$
|v^\sharp(x)|^2\le |v(x)|^2
$
since $|\phi(t)|\le |t|$, while
$
-\log(1-x^2)\ge0.
$
Therefore
\[
-\frac12\int_{-1}^1
|v^\sharp(x)|^2\log(1-x^2)\,dx
\le
-\frac12\int_{-1}^1
|v(x)|^2\log(1-x^2)\,dx.
\]
Combining these estimates, we obtain
$
\mathcal L(v^\sharp)\le\mathcal L(v).
$
Thus $\mathcal L$ is Markovian.
Since $\mathcal L$ is closable, its closure
$(\bar{\mathcal L},\mathfrak D(\bar{\mathcal L}))$
is also Markovian and hence is a Dirichlet form by
\cite[Theorem 3.1.1]{FOT11}.
In particular, the normal contraction $v\mapsto |v|$ operates on
$\bar{\mathcal L}$. Hence a minimizer $v_0$ of the Rayleigh quotient
$
\bar{\mathcal L}(v)/\|v\|^2
$
may be chosen to be nonnegative.

We next show that $(\bar{\mathcal{L}},\mathfrak{D}(\bar{\mathcal{L}}))$ is irreducible.
Recall that a Dirichlet form is called irreducible if every invariant set is trivial,
that is, either it or its complement has measure zero.
Let $A\subset(-1,1)$ be a $\bar{\mathcal{L}}$-invariant set.
By the characterization of invariant sets in \cite[p.~173]{LSV09},
for every $u\in\mathfrak{D}(\bar{\mathcal{L}})$, 
$
1_Au,\ 1_{A^c}u\in\mathfrak{D}(\bar{\mathcal{L}})
$
and
$
\bar{\mathcal{L}}(1_Au,1_{A^c}u)=0,
$
where $A^c=(-1,1)\setminus A$.
For $u,v\in\mathfrak{D}(\mathcal{L})$, polarization of
\eqref{EQ_404} gives
\[
\mathcal{L}(u,v)
=
\frac{1}{4}
\int_{-1}^1
\int_{-1}^1
\frac{(u(x)-u(y))\overline{(v(x)-v(y))}}{|x-y|}
\,dx\,dy
-
\frac{1}{2}
\int_{-1}^1
u(x)\overline{v(x)}\log(1-x^2)\,dx .
\]
Since the two quadratic terms in \eqref{EQ_404} are nonnegative,
convergence in the $\mathcal L$-form norm implies convergence of
each term separately.
Hence the above polarized identity extends to
$\mathfrak{D}(\bar{\mathcal{L}})$ by approximation in the form norm.
Therefore, for every real-valued
$u\in\mathfrak{D}(\mathcal{L})$,
\[
\bar{\mathcal{L}}(1_Au,1_{A^c}u)
=
-\frac{1}{2}
\int_A\int_{A^c}
\frac{u(x)u(y)}{|x-y|}
\,dx\,dy
=0.
\]
Taking, for instance, $u(x)=1-x^2$, which belongs to
$\mathfrak{D}(\mathcal{L})$ and is strictly positive on $(-1,1)$,
the strict positivity of the kernel $|x-y|^{-1}$ shows that the above
identity is impossible if both $A$ and $A^c$ have positive measure.
Thus either $A$ or $A^c$ has measure zero, and
$\bar{\mathcal{L}}$ is irreducible.

Since $(\bar{\mathcal{L}},\mathfrak{D}(\bar{\mathcal{L}}))$
is a Dirichlet form, its associated semigroup is positivity preserving.
By \cite[Theorem 1.4]{LSV09}, the irreducibility implies that
the associated semigroup $\{e^{-tT}\mid t>0\}$ is positivity improving.
In other words, for every nonnegative function $f\in L^2(-1,1)$
with $f\not\equiv0$, one has $(e^{-tT}f)(x)>0$ for a.e. $x\in(-1,1)$.
In particular, the nonnegative eigenfunction $v_0$ of $T$ is positive almost everywhere.
\medskip

\subsection{Simplicity of the lowest eigenvalue and parity of eigenfunctions}~
Since $\bar{\mathcal L}$ is real, the associated operator $T$
commutes with complex conjugation. 
Therefore, it suffices to consider
real-valued eigenfunctions.
Every real-valued eigenfunction corresponding to the lowest eigenvalue
has a fixed sign.
Indeed, if such an eigenfunction $u$ changed sign on sets of positive measure,
then the strict positivity of the jumping kernel in \eqref{EQ_404} would give
$\bar{\mathcal L}(|u|)<\bar{\mathcal L}(u)$,
contradicting the minimality of $u$.
After changing its sign, each such nonzero eigenfunction is nonnegative
and hence positive almost everywhere by the preceding subsection.
Thus two such eigenfunctions cannot be orthogonal.
Hence the eigenspace corresponding to the lowest eigenvalue of $T$
is one-dimensional, and the lowest eigenvalue is simple.

For $0<a<(1/2)\log 2$, the prime sum in \eqref{EQ_405} is empty.
After subtracting the scalar term
$(\log(1/a)-(2A+1))\|v\|_{L^2}^2$,
the associated operator is
\[
T+R_a,
\qquad
(R_aw)(x)
=
-a\int_{-1}^1 r''(a(x-y))w(y)\,dy .
\]
Since $r''$ is continuous near the origin, $R_a$ is a bounded
self-adjoint operator and
$\|R_a\|=O(a)$ as $a\to0+$.
Let $\mu_1<\mu_2$ be the two lowest eigenvalues of $T$.
By the min--max principle,
\[
\lambda_1(T+R_a)\leq\mu_1+\|R_a\|,
\qquad
\lambda_2(T+R_a)\geq\mu_2-\|R_a\|.
\]
Hence, for sufficiently small $a>0$, the lowest eigenvalue of
$T+R_a$, and therefore $\lambda_a$, is simple.
Moreover, since $\bar{\mathcal L}$ is invariant under parity and $r$ is even,
both $T$ and $R_a$ commute with the parity operator $J$. 
The unique positive ground state of $T$ is therefore even.
Consequently, the bottom of the spectrum of $T$ on the odd subspace
is at least $\mu_2$, whereas the bottoms of the even and odd parts of
$T+R_a$ are respectively at most
$\mu_1+\|R_a\|$ and at least $\mu_2-\|R_a\|$.
Thus, for sufficiently small $a>0$, the lowest eigenspace is even.
This is consistent with the general results in
\cite[Chapter VIII, Section 3, Theorems 3.6 and 3.15]{Ka95}.

%
\section{Proof of Theorem~\ref{thm_5}} \label{section_6}
%

\subsection{Symmetry of the minimal operator} 

Choose $\lambda<\lambda_a$, and let $\mathcal{H}(T_a)$ denote the Hilbert space obtained as the completion of $C_c^\infty(-a,a)$ with respect to the norm $\|\cdot\|_{T_a}$ defined by \eqref{EQ_109}. Define
\[
\langle v_1,v_2\rangle_{T_a}
:=\langle T_a v_1,v_2\rangle_{L^2}
\]
so that $\|v\|_{T_a}^2=\langle v,v\rangle_{T_a}$. 
Let $\D_a$ be the operator on $\mathcal{H}(T_a)$ 
with domain \eqref{EQ_110} acting as $\D_a=i\,d/dx$. 
Since the canonical map 
$C_c^\infty(-a,a)\to \mathcal{H}(T_a)$
is injective and the image is dense, the operator $\D_a$ is densely defined.

\begin{lemma} \label{lem_7_1}
The operator $\D_a$ is symmetric on $\mathcal{H}(T_a)$.
\end{lemma}
\begin{proof} 
We show that
$\langle \D_a u, v \rangle_{T_a}
=
\langle u, \D_a v \rangle_{T_a}$
for all $u,v \in \mathfrak{D}(\D_a)$.
By the definition of the inner product on $\mathcal{H}(T_a)$ and the self-adjointness of $A_a$,
we have
\[
\langle \D_a u, v \rangle_{T_a}
=
\langle A_a(iu'), v \rangle_{L^2}
-\lambda \langle iu', v \rangle_{L^2}
=
\langle iu', A_a v \rangle_{L^2}
-\lambda \langle iu', v \rangle_{L^2}.
\]
By \eqref{EQ_210}, $A_a v$ is the restriction to $(-a,a)$ of
the convolution $(-g'')*v$.
Since $-g''$ is a distribution and
$v\in \mathfrak{D}(\D_a)=C_c^\infty(-a,a)$,
this convolution is smooth on $\R$. 
Using this fact together with the boundary conditions
$u(a)=u(-a)=0$,
integration by parts gives
$
\langle iu', A_a v \rangle_{L^2}
=
\langle u, i(A_a v)' \rangle_{L^2},
$
and
$
\langle iu', v \rangle_{L^2}
=
\langle u, iv' \rangle_{L^2}. 
$
Furthermore, using \eqref{EQ_210}
and the boundary conditions $v(a)=v(-a)=0$,
we obtain by integration by parts that
\[
(A_a v)'(x)
=
\frac{d}{dx}\int_{-a}^{a} k(x-y)v(y)\,dy
=
-\int_{-a}^{a}\frac{d}{dy}k(x-y)v(y)\,dy
=
\int_{-a}^{a}k(x-y)v'(y)\,dy.
\]
Therefore,
$\langle u, i(A_a v)' \rangle_{L^2}
=
\langle u, A_a(iv') \rangle_{L^2}$,
and hence
\[
\langle \D_a u, v \rangle_{T_a}
=
\langle u, A_a(iv') \rangle_{L^2}
-\lambda \langle u, iv' \rangle_{L^2}
=
\langle A_a u, iv' \rangle_{L^2}
-\lambda \langle u, iv' \rangle_{L^2}
=
\langle u, \D_a v \rangle_{T_a}.
\]
This proves the desired identity.
\end{proof}

\subsection{Eigenfunctions of the adjoint operator} \label{section_6_3}

We now construct the eigenfunctions of the adjoint operator $\D_a^\ast$
directly.
We have 
$
\|v\|_{L^2}
\leq
(\lambda_a-\lambda)^{-1/2}\|v\|_{T_a}
$
for $v\in C_c^\infty(-a,a)$.
For each $z\in\C$, we also have
\[
|\widehat v(\bar z)|
\leq
(2a)^{1/2}e^{a|\Im z|}\|v\|_{L^2}
\ll_{a,z}
\|v\|_{T_a}.
\]
Hence, by the Riesz representation theorem, there exists a unique
$v_z\in\mathcal H(T_a)$ such that
\[
\langle v,v_z\rangle_{T_a}=\widehat v(\bar z)
\]
for all $v\in\mathcal H(T_a)$.
The point-evaluation functional is nonzero, so $v_z\neq0$. 
For every $v\in\mathfrak D(\D_a)=C_c^\infty(-a,a)$, 
integration by parts gives $\widehat{iv'}(\bar z)=\bar z\,\widehat v(\bar z)$.
Therefore,
\[
\langle \D_a v,v_z\rangle_{T_a}
=
\widehat{iv'}(\bar z)
=
\bar z\,\widehat v(\bar z)
=
\langle v,zv_z\rangle_{T_a}.
\]
By the definition of the adjoint operator, it follows that
\[
v_z\in\mathfrak D(\D_a^\ast),
\qquad
\D_a^\ast v_z=zv_z.
\]

We next show that the eigenspace corresponding to $z$ is one-dimensional.
Let $f\in\ker(\D_a^\ast-z)$ and define a linear functional on
$C_c^\infty(-a,a)$ by
\[
F_f(v):=\langle v,f\rangle_{T_a}.
\]
This functional is a distribution on $(-a,a)$.
Indeed, by the distributional formula \eqref{EQ_209},
the norm $\|\cdot\|_{T_a}$ restricted to $C_c^\infty(-a,a)$
is continuous with respect to the usual topology, and hence 
$|F_f(v)|\leq \|f\|_{T_a}\|v\|_{T_a}$ 
shows the required continuity.
Since $\D_a^\ast f=zf$, we have
\[
iF_f(v')
=
\langle iv',f\rangle_{T_a}
=
\langle \D_a v,f\rangle_{T_a}
=
\langle v,\D_a^\ast f\rangle_{T_a}
=
\bar z\,F_f(v).
\]
Therefore, in the sense of distributions,
\[
F_f'=i\bar z\,F_f.
\]
Equivalently, $\bigl(e^{-i\bar z x}F_f\bigr)'=0$. 
Since a distribution with zero derivative on the interval $(-a,a)$ is constant,
there exists $C\in\C$ such that
\[
F_f(v)
=
C\int_{-a}^a v(x)e^{i\bar z x}\,dx
=
C\widehat v(\bar z)
\]
for all $v\in C_c^\infty(-a,a)$.
By the defining identity for $v_z$, 
$\langle v,f\rangle_{T_a}
=
C\langle v,v_z\rangle_{T_a}$ 
for all $v\in C_c^\infty(-a,a)$.
Since $C_c^\infty(-a,a)$ is dense in $\mathcal H(T_a)$,
$f$ is a scalar multiple of $v_z$.
Consequently, $
\ker(\D_a^\ast-z)=\C v_z$  for every $z\in\C$.

\subsection{Deficiency indices}

Taking $z=\pm i$ in the preceding equality, we immediately obtain the following result.

\begin{lemma}  \label{lem_6_2}
$\D_a$ has the deficiency indices $(1,1)$.
\end{lemma}

We also mention a concrete representation of the vectors $v_z$.
Let
\[
e_z(x):=e^{-izx}.
\]
Since $T_a\geq(\lambda_a-\lambda)I>0$, the inverse $T_a^{-1}$ is a bounded
operator on $L^2(-a,a)$. 
By Theorem~\ref{thm_1} and the density result
established in Section~\ref{section_3},
$\mathcal H(T_a)$ is identified with the form domain of $T_a$.
Set
\[
w_z:=T_a^{-1}e_z\in\mathfrak D(T_a).
\]
By the representation theorem for the closed quadratic form associated with $T_a$,
for every $v\in\mathcal H(T_a)$ we have
\[
\langle v,w_z\rangle_{T_a}
=
\langle v,T_aw_z\rangle_{L^2}
=
\langle v,e_z\rangle_{L^2}
=
\widehat v(\bar z).
\]
The uniqueness in the Riesz representation theorem therefore gives
\[
v_z=T_a^{-1}e_z.
\]
In particular, $v_z\in\mathfrak D(T_a)$.

\subsection{Self-adjoint extensions of the minimal operator}

We compute the self-adjoint extensions of $\D_a$
using von Neumann's theory of deficiency indices
(\cite[Section X.1]{RS75}, \cite[Section 3.2]{Sch12}).
Choose basis vectors $v_\pm$ of the deficiency spaces
${\rm Ker}(\D_a^\ast\mp i)$ such that $\|v_+\|_{T_a}=\|v_-\|_{T_a}$. 
By the preceding subsection, each $v_\pm$ is a nonzero scalar multiple of
$v_{\pm i}$.
Thus, every self-adjoint extension
$\overline{\mathscr{D}}_{a,\theta}$
parametrized by $\theta\in[0,2\pi)$ has domain
\[
\aligned
\mathfrak{D}(\overline{\mathscr{D}}_{a,\theta})
&=
\{v_0+c(v_++e^{i\theta}v_-)\mid
v_0\in\mathfrak D(\overline{\D}_a),~c\in\C\}
\\
\subset&\,\,
\mathfrak D(\D_a^\ast)
=
\{v_0+a v_++b v_-\mid
v_0\in\mathfrak D(\overline{\D}_a),~a,b\in\C\}.
\endaligned
\]
These domains are characterized by the boundary form
\begin{equation} \label{EQ_601}
W(u,v)
:=
\langle \D_a^\ast u,v\rangle_{T_a}
-
\langle u,\D_a^\ast v\rangle_{T_a},
\end{equation}
namely, $v\in\mathfrak D(\D_a^\ast)$ belongs to
$\mathfrak D(\overline{\mathscr{D}}_{a,\theta})$ if and only if
\[
W(v,w_\theta)=0,
\qquad
w_\theta:=v_++e^{i\theta}v_-.
\]
The action of $\overline{\mathscr{D}}_{a,\theta}$ is given by
\[
\overline{\mathscr{D}}_{a,\theta}
\Bigl(v_0+c(v_++e^{i\theta}v_-)\Bigr)
:=
\overline{\D}_a v_0+ic(v_+-e^{i\theta}v_-),
\]
where $\overline{\D}_a$ is the closure of $\D_a$.

\subsection{Proof of Theorem~\ref{thm_5}} \label{section_6_5}

Let $v_\pm$ and $w_\theta$ be as in the preceding subsection.
Because $v_\pm\in\mathfrak D(T_a)\subset L^2(-a,a)\subset L^1(-a,a)$,
the functions
\[
z\longmapsto\widehat{v_\pm}(z)
=
\int_{-a}^a v_\pm(x)e^{izx}\,dx
\]
are entire.
Hence $W(a,\theta;z)$ defined in \eqref{EQ_111} is entire. 
For $z\in\C$, we have
\[
\D_a^\ast v_{\bar z}=\bar z v_{\bar z},
\qquad
\D_a^\ast w_\theta=i v_+-i e^{i\theta}v_-.
\]
Substituting these identities into the boundary form \eqref{EQ_601} and taking
complex conjugates, we obtain
\begin{equation} \label{EQ_602}
\aligned
\overline{W(v_{\bar z},w_\theta)}
&=
(z-i)\langle v_+,v_{\bar z}\rangle_{T_a}
+e^{i\theta}(z+i)\langle v_-,v_{\bar z}\rangle_{T_a}
\\
&=
(z-i)\widehat{v_+}(z)
+e^{i\theta}(z+i)\widehat{v_-}(z)
\\
&=
W(a,\theta;z),
\endaligned
\end{equation}
where we used the defining identity for $v_{\bar z}$ in the second equality.
For later use, we also record the equivalent identity
\begin{equation} \label{EQ_603}
\langle v_z,v_\pm\rangle_{T_a}
=
\overline{\langle v_\pm,v_z\rangle_{T_a}}
=
\overline{\widehat{v_\pm}(\bar z)}
=
\int_{-a}^a \overline{v_\pm(x)}e^{-izx}\,dx.
\end{equation}

Suppose first that $W(a,\theta;z)=0$.
By \eqref{EQ_602},
$W(v_{\bar z},w_\theta)=0$, and hence
$v_{\bar z}\in\mathfrak D(\overline{\mathscr D}_{a,\theta})$.
Since $v_{\bar z}\neq0$ and
$\D_a^\ast v_{\bar z}=\bar z v_{\bar z}$, we obtain
$
\overline{\mathscr D}_{a,\theta}v_{\bar z}
=
\bar z v_{\bar z}$.
The operator $\overline{\mathscr D}_{a,\theta}$ is self-adjoint, so $\bar z$,
and hence $z$, is real.
Thus every zero of $W(a,\theta;z)$ is real.

Conversely, let $z\in\R$ be an eigenvalue of
$\overline{\mathscr D}_{a,\theta}$.
Any corresponding eigenvector belongs to
$\ker(\D_a^\ast-z)=\C v_z$ by the one-dimensionality of the eigenspace proved above.
Therefore $v_z\in\mathfrak D(\overline{\mathscr D}_{a,\theta})$, so
$W(v_z,w_\theta)=0$.
Since $z=\bar z$, \eqref{EQ_602} gives
$W(a,\theta;z)=0$.
Hence the zeros of $W(a,\theta;z)$ are precisely the eigenvalues of
$\overline{\mathscr D}_{a,\theta}$. 

It remains to prove that these zeros are simple.
For $u\in C_c^\infty(-a,a)$, the definition of $T_a$ gives
\[
\|u\|_{H^1(-a,a)}^2
\leq
\frac{1}{\lambda_a-\lambda}
\left(\|u\|_{T_a}^2+\|\D_a u\|_{T_a}^2\right).
\]
Consequently,
$\mathfrak D(\overline{\D}_a)\subset H_0^1(-a,a)$,
and $\overline{\D}_a u=iu'$ on this domain.
Since a solution of $iu'=zu$ satisfying the zero
boundary conditions must vanish identically,
$\overline{\D}_a$ has no eigenvalues.

Let $t\in\R$ be a zero of $W(a,\theta;z)$.
By the preceding argument,
$v_t\in\mathfrak D(\overline{\mathscr D}_{a,\theta})$.
Using the von Neumann decomposition, write
\[
v_t=v_0+c_t w_\theta,\qquad
v_0\in\mathfrak D(\overline{\D}_a),\quad c_t\in\C.
\]
Here $c_t\ne0$, since otherwise $v_t$ would be
an eigenvector of $\overline{\D}_a$.
For every $z\in\C$, the boundary form satisfies
\[
\begin{aligned}
\bar c_t W(v_{\bar z},w_\theta)
=W(v_{\bar z},v_t)
=(\bar z-t)\langle v_{\bar z},v_t\rangle_{T_a}.
\end{aligned}
\]
Taking complex conjugates and using \eqref{EQ_602}
and the defining identity for $v_{\bar z}$, we obtain
\[
W(a,\theta;z)
=\frac{z-t}{c_t}\widehat{v_t}(z).
\]
Since $\widehat{v_t}(t)=\|v_t\|_{T_a}^2>0$,
it follows that
\[
\left.\frac{\partial}{\partial z}
W(a,\theta;z)\right|_{z=t}
=\frac{\|v_t\|_{T_a}^2}{c_t}\ne0.
\]
Thus every zero of $W(a,\theta;z)$ is simple. \hfill $\Box$

%
\section{Heuristic considerations on the formula in Corollary~\ref{cor_6}} \label{section_7}
%

In this section, we give heuristic considerations on the formula \eqref{EQ_112} in Corollary~\ref{cor_6}. 
Some of the arguments below are heuristic and do not constitute rigorous proofs. 

Throughout this section, we assume RH. 
By Weil's positivity criterion, we have $Q_W \geq 0$ and hence $A_a>0$ for all $a>0$.
Under this assumption,
for each $a>0$ we may take $\lambda=0$ in the definition \eqref{EQ_109} of $T_a=T_{a,\lambda}$.
Hence we work with $A_a$ in place of $T_a$, and consider the Hilbert spaces
$\mathcal{H}(A_a):=\mathcal{H}(T_{a,0})$.

\subsection{The Hilbert Space $\mathcal{H}(A_\infty)$}

In order to study the behavior of the differential operator $\D_a$ 
on the Hilbert space $\mathcal{H}(A_a)$ as $a \to \infty$, 
we first introduce the Hilbert space $\mathcal{H}(A_\infty)$ and 
the corresponding differential operator $\D_\infty$. 
Here $A_\infty$ is only a formal symbol and does not denote an actual operator, 
but we shall use the notation $A_\infty>0$ only as shorthand for $A_a>0$ for every $a>0$.

Let $\mathfrak{D}(Q_W)$ denote the completion of $C_c^\infty(\mathbb R)$ 
with respect to the form norm
\[
\|v\|_{Q_W}^2 := Q_W(v) + \|v\|_{L^2}^2.
\]
This space is a Hilbert space with respect to the form norm.
Here and below, this completion is understood abstractly and is not
identified with a subspace of $L^2(\mathbb R)$.
Accordingly, formulas involving the integral representation of $Q_W$
or the values $\widehat v(\gamma)$ are first understood for
$v\in C_c^\infty(\mathbb R)$ and then interpreted by continuous extension
whenever they are used on the completion.  
The subspace 
\[
\mathcal{N} := \{ v \in \mathfrak{D}(Q_W) \mid Q_W(v) = 0 \}
\]
is closed in $\mathfrak{D}(Q_W)$ with respect to the form norm.
We then define $\mathcal H(A_\infty)$ as the completion of
$\mathfrak D(Q_W)/\mathcal N$
with respect to the norm
\[
\|[v]\|_{A_\infty}^2:=Q_W(v).
\]
For $v \in \mathfrak{D}(Q_W)$, we denote its image in $\mathcal{H}(A_\infty)$ by $[v]$.
The Hilbert space $\mathcal{H}(A_\infty)$ can be identified with the completion $\mathcal{H}_W$ 
of $C_c^\infty(\mathbb{R})$ with respect to $Q_W(v)$ studied in \cite{Su25}.

Since $Q_W(v) \neq 0$ for any nonzero $v \in C_c^\infty(\mathbb{R})$, 
the map $v \mapsto [v]$ embeds $C_c^\infty(\mathbb{R})$ injectively into $\mathcal{H}(A_\infty)$, 
and its image is dense.
Hence, by fixing representatives on $C_c^\infty(\mathbb R)$,
we define a densely defined operator $\D_\infty$ on
$\mathcal H(A_\infty)$ by
\[
\D_\infty [v] := [\D_\infty v] = [iv'].
\]
This is well-defined since $[u]=[v]$ for $u,v \in C_c^\infty(\mathbb{R})$ implies $u=v$.

We now verify that $\D_\infty$ is symmetric on $\mathcal H(A_\infty)$ and hence closable. 
If $Q_W(w)=0$, it follows that $Q_W(w,v)=0$ for all $v$ 
by the assumption $Q_W \geq 0$
and 
the Cauchy--Schwarz inequality $|Q_W(u,v)|^2 \le Q_W(u) Q_W(v)$. 
This implies $
Q_W(u+w,v)=Q_W(u,v)$ and $Q_W(u,v+w)=Q_W(u,v)$. 
Hence the sesquilinear form
\[
Q_W([u],[v]):=Q_W(u,v)
\]
is well-defined on $\mathcal{H}(A_\infty)$.
For $u,v\in C_c^\infty(\mathbb R)$, it is represented by
\begin{equation} \label{EQ_701}
Q_W([u],[v])
=
\int_{-\infty}^{\infty} \int_{-\infty}^{\infty}
(-g''(x-y))\,u(y)\,\overline{v(x)}\,dx\,dy
\end{equation}
(cf. \eqref{EQ_209}), and this identity is extended to
$\mathcal{H}(A_\infty)$ by continuity. 
Moreover, for $t\in\mathbb R$, the translation
\[
(\tau_tu)(x):=u(x-t),
\qquad u\in C_c^\infty(\mathbb R),
\]
preserves $Q_W$ by \eqref{EQ_701}.
Hence it induces an isometry on the dense subspace 
$\{[u]\mid u\in C_c^\infty(\mathbb R)\}$ 
of
$\mathcal H(A_\infty)$ and therefore extends uniquely to an isometry
$\tau_t$ on $\mathcal H(A_\infty)$. 
Furthermore, by \eqref{EQ_701}, the norm in $\mathcal{H}(A_\infty)$ 
is invariant under translations:
\begin{equation} \label{EQ_702}
\|[v]\|_{A_\infty}
=
\|\tau_t [v]\|_{A_\infty}.
\end{equation}
Since $\D_\infty$ commutes with convolution, it follows from \eqref{EQ_701} that
\[
Q_W(\D_\infty [u], [v]) = Q_W([u], \D_\infty [v])
\]
for $u,v \in C_c^\infty(\mathbb{R})$. 
Hence $\D_\infty$ is symmetric with respect to
$\langle\cdot,\cdot\rangle_{A_\infty}$.
Since every symmetric operator is closable,
its closure will be denoted by $\bar{\D}_\infty$.
Thus, for $[v]\in\mathfrak D(\bar{\D}_\infty)$,
one can choose a sequence $v_n\in C_c^\infty(\R)$ such that
$[v_n]\to [v]$ and
$\D_\infty[v_n]=[iv_n']\to \bar{\D}_\infty[v]$
in $\mathcal H(A_\infty)$.

\subsection{Isomorphism with a de Branges space}

We now recall the isometric isomorphism between $\mathcal{H}(A_\infty)$ 
(denoted by $\mathcal{H}_W$ in \cite{Su25}) 
and a certain de Branges space $\mathcal{B}$ established in \cite{Su25}.
Here it is sufficient to understand $\mathcal{B}$ 
as a reproducing kernel Hilbert space consisting of entire functions.

Let $M$ be the operator of multiplication by an independent variable on $\mathcal{B}$ 
defined by $\mathfrak{D}(M)=\{ F \in \mathcal{B} \mid zF(z) \in \mathcal{B}\}$ 
and $(MF)(z)=zF(z)$. 
The operator $M$ is a densely defined closed symmetric operator with deficiency indices $(1,1)$, 
and it admits a one-parameter family of self-adjoint extensions parametrized 
by $\psi \in [0,\pi)$. In particular, the self-adjoint extension $M_{\pi/2}$ has purely discrete spectrum 
$\Gamma$, the set of zeros of $\xi(1/2-iz)$ in Section~\ref{section_3}, 
and one can choose an orthonormal basis of $\mathcal{B}$ 
consisting of the corresponding eigenfunctions 
$EF_\gamma$ for $\gamma \in \Gamma$ \cite[Section 6]{Su25}, 
where $E(z)=\xi(1/2-iz)+\xi'(1/2-iz)$. 
Let $\mathcal{K}:=\mathcal{B}/E=\{F/E \mid F \in \mathcal{B}\}$ be the Hilbert space 
defined by 
$\langle F/E, G/E \rangle_{\mathcal{K}}:=\langle F, G \rangle_{\mathcal{B}}$. 
Via this isomorphism, we identify $\mathcal{B}$ with $\mathcal{K}$ 
and may regard operators on $\mathcal{B}$ as operators on $\mathcal{K}$ for simplicity.
In this setting, defining $U: \mathcal{H}(A_\infty) \to \mathcal{K}$ by
\begin{equation} \label{EQ_703}
U([v])
:= \sum_{\gamma \in \Gamma} \sqrt{m_\gamma}\, \widehat{v}(\gamma)  F_\gamma 
\end{equation}
yields an isometric isomorphism of Hilbert spaces \cite[Theorems  1.1 and 1.4]{Su25}:
\begin{equation} \label{EQ_704}
(Q_W(v)=)~\| [v] \|_{A_\infty}^2 = \| U([v]) \|_{\mathcal{K}}^2.
\end{equation}
This $U$ is precisely the operator denoted 
by $\pi^{-1/2}\widehat{\mathcal{P}_D} = \pi^{-1/2}\widehat{\mathcal{P}} D$ in \cite{Su25}. 

\subsection{Correspondence between $\D_\infty$ and $M$}  

Via the isomorphism $U: \mathcal{H}(A_\infty) \to \mathcal{K}$,
the operator $M$ induces a closed symmetric operator on
$\mathcal{H}(A_\infty)$.
However, its relation to the differential operator
$\bar{\D}_\infty$ is not immediately clear. 
On the other hand, the relation between the differential operator $\bar{\D}_\infty$ and 
the self-adjoint extension $M_{\pi/2}$ of $M$ can be described in a simple way as follows.

If $v \in C_c^{\infty}(\R)$, then $\widehat{v}$ is rapidly decreasing. 
Hence $U([v])$ belongs to the domain of $M_{\pi/2}$.
For each $\gamma \in \Gamma$, $F_\gamma$ is an eigenfunction of
$M_{\pi/2}$ with eigenvalue $\gamma$, and therefore
\[
\aligned 
M_{\pi/2} U([v])
& = \sum_{\gamma \in \Gamma} \sqrt{m_\gamma}\, \gamma \widehat{v}(\gamma) F_\gamma 
   = \sum_{\gamma \in \Gamma} \sqrt{m_\gamma}\, \widehat{(iv')}(\gamma) F_\gamma = U([iv'])= U(\D_\infty[v]). 
\endaligned 
\]
Let $\mathsf{D}_{\pi/2}:=U^{-1}M_{\pi/2} U$.  
Then the above computation shows that
$\mathsf{D}_{\pi/2}[v] = [iv'] = \D_\infty[v]$ 
for $v \in C_c^\infty(\R)$. 
Since $\mathsf D_{\pi/2}$ is closed and extends
$\D_\infty$, it follows that
\[
\mathsf D_{\pi/2}[v]
=
\bar{\D}_\infty[v]
\]
for all $[v]\in\mathfrak D(\bar{\D}_\infty)$. 
Thus, $\mathsf D_{\pi/2}$ is a self-adjoint extension of $\bar{\D}_\infty$.
\medskip

\noindent
{\bf Remark.}~
More generally, for any self-adjoint extension $M_\psi$ of $M$,
one may define $\mathsf D_\psi:=U^{-1}M_\psi U$. 
Since $M_\psi$ is unitarily equivalent to $\mathsf D_\psi$,
the operator $\mathsf D_\psi$ is self-adjoint.
It is natural to ask whether $\mathsf D_\psi$
can be characterized directly as a self-adjoint extension of
$\bar{\D}_\infty$.
We do not pursue this question here.

\subsection{An operator unitarily equivalent to $M$}

We now consider the operator on $\mathcal H(A_\infty)$ induced by $M$ through $U$. 
More precisely, we define
\[
\mathsf{D} := U^{-1} M U.
\]
By construction, $\mathsf D$ and $M$ are unitarily equivalent via $U$, 
and have the same deficiency indices. 
In particular, the deficiency indices of $\mathsf{D}$ are $(1,1)$.
\medskip

Writing the Fourier transform as $\mathscr{F}v=\widehat{v}$,
there exists a subspace
\[
V(0)\subset L^2(0,\infty)
\]
such that the maps
$\mathcal{K}\ni F\mapsto v_F:=\mathscr{F}^{-1}F\in V(0)$ and
$V(0)\ni v\mapsto[v]\in\mathcal{H}(A_\infty)$
are isomorphisms \cite[Theorem 5.5]{Su25}. 
Let
\[
D:=\mathscr{F}^{-1}M\mathscr{F},
\qquad
\mathfrak D(D):=\mathscr{F}^{-1}\bigl(\mathfrak D(M)\bigr).
\]
Define $\iota:V(0)\to\mathcal{H}(A_\infty)$ by $\iota(v):=[v]$. 
Moreover, 
\[
\|\mathscr{F}v\|_{\mathcal K}^2
=
2\pi\|v\|_{L^2}^2
=
\pi\|[v]\|_{A_\infty}^2
\qquad
(v\in V(0)).
\]
Hence the unitary map $U$ in \eqref{EQ_703} satisfies 
$U=\pi^{-1/2}\mathscr{F}\iota^{-1}$. 
The scalar factor cancels in the conjugation, and therefore
\[
\mathsf D
=
U^{-1}MU
=
\iota\mathscr{F}^{-1}M\mathscr{F}\iota^{-1}
=
\iota D\iota^{-1}.
\]
If $F\in\mathfrak D(M)$ and $F=\mathscr{F}v$, then
\[
Dv
=
\mathscr{F}^{-1}(MF)
=
\mathscr{F}^{-1}(zF)
=
iv'.
\]
Consequently, for $v\in\mathfrak D(D)$,
\[
\mathsf D[v]
=
\iota(Dv)
=
\iota(iv')
=
[iv'].
\]

Since $\iota:V(0)\to\mathcal H(A_\infty)$, $v\mapsto[v]$, 
is an isomorphism, every element $[u]\in\mathcal H(A_\infty)$ has a unique representative
$v\in V(0)$ in the sense that $[u]=[v]$.
If $[u]\in\mathfrak D(\mathsf D)$, then this representative belongs to
$\mathfrak D(D)$, and
$\mathsf D[u]=[Dv]$. 
In view of this identification, the relation between
\[
C_c^\infty(\R) \quad \text{and} \quad \mathfrak{D}(\mathsf{D})
\]
is not clear.
Indeed, for $u\in C_c^\infty(\R)$, let
$v=\iota^{-1}[u]\in V(0)$.
It is not clear whether
$\mathscr Fv\in\mathfrak D(M)$,
or equivalently whether $v\in\mathfrak D(D)$.

\subsection{Infinitesimal generator of translations ($a=\infty$)} \label{section_0531}

The previously defined family $(\tau_t)_{t \in \R}$ forms a strongly continuous one-parameter unitary group. 
Indeed, $\tau_{t}\tau_{s}=\tau_{t+s}$ is clear from the definition, 
and each $\tau_t$ is a unitary operator by \eqref{EQ_702}. 
It remains to verify strong continuity, namely 
$\lim_{t \to 0} \|\tau_t [v] - [v]\|_{A_\infty} = 0$. 
To this end, using the isometric property of $U$, we estimate
\[
\| U(\tau_t [v] - [v]) \|_{\mathcal{K}}^2 
= \sum_{\gamma \in \Gamma} m_\gamma |e^{it\gamma}-1|^2 |\widehat{v}(\gamma)|^2.
\]
Since $A_\infty>0$ implies RH,
all elements of $\Gamma$ are real. 
Since $|e^{it\gamma}-1|^2\le4$,
the dominated convergence theorem implies that the right-hand side tends to $0$ as $t \to 0$. 
Hence $\tau_t$ is strongly continuous with respect to the norm of $\mathcal{H}(A_\infty)$.

Therefore, 
there exists a self-adjoint operator $\overline{\mathscr{D}}$ on $\mathcal{H}(A_\infty)$ such that
\[
\tau_t = e^{it\overline{\mathscr{D}}} \quad (t \in \R)
\]
by Stone's theorem, where 
\[
\mathfrak D(\overline{\mathscr D})
=
\left\{
[v]\in\mathcal H(A_\infty) \left|~
\lim_{\epsilon\to0}
\frac{\tau_\epsilon[v]-[v]}{\epsilon}
\text{ exists}
\right.
\right\},
\quad 
\overline{\mathscr{D}}[v]=-i\lim_{\epsilon \to 0} \frac{\tau_\epsilon [v]-[v]}{\epsilon}.
\]
For $[v] \in \mathfrak{D}(\mathsf{D})$, 
we have 
\[
\overline{\mathscr{D}}[v] = [iv'] = \mathsf{D}[v]
\]
by considering the action on the $\mathcal{K}$ side. 
Thus, the self-adjoint operator
$\overline{\mathscr D}$
extends the closed symmetric operator
$\mathsf D$,
whose deficiency indices are $(1,1)$.
Moreover,
\[
U(\overline{\mathscr{D}}[v]) = M_{\pi/2} U([v])
\]
holds, that is,
$
\overline{\D} =
\mathsf{D}_{\pi/2}.
$
Indeed,
\[
\aligned 
(U\overline{\mathscr{D}}[v])(z) 
&=-i\lim_{\epsilon \to 0} 
\frac{(U\tau_\epsilon [v])(z)-(U[v])(z)}{\epsilon} \\
&=-i\lim_{\epsilon \to 0} 
\sum_{\gamma \in \Gamma} \sqrt{m_\gamma}\, 
\frac{e^{i\gamma \epsilon}-1}{\epsilon} \, \widehat{v}(\gamma)  F_\gamma \\
& = \sum_{\gamma \in \Gamma} \sqrt{m_\gamma}\, \gamma \widehat{v}(\gamma)  F_\gamma 
= (M_{\pi/2} U[v])(z). 
\endaligned 
\]
Viewed on $\mathcal H(A_\infty)$,
$\overline{\mathscr D}$ is a self-adjoint extension of $\mathsf D$
and is the infinitesimal generator of the translation group. 
Moreover, the equality $\overline{\D} = \mathsf{D}_{\pi/2}$ asserts that 
a Hilbert--P{\'o}lya operator is realized as the infinitesimal generator of the translation group 
under the assumption $A_\infty>0$ (RH).

\medskip

Since $\overline{\mathscr{D}}$ acts as $v \mapsto iv'$ on $C_c^\infty(-a,a)$, 
to approximate the behavior as $a \to \infty$, 
it is natural to consider the minimal operator $\D_a$ for each $a>0$. 
Considering the self-adjoint extensions $\mathscr{D}_{a, \theta}$ of $\mathscr{D}_a$, it is expected that by choosing $\theta=\theta(a)$ appropriately,
\[
\overline{\D}_{a, \theta}~\to~\overline{\mathscr{D}} \quad (a \to \infty)
\]
in a suitable sense after removing the contribution from $\mathcal N$. 

\subsection{Operator picture via embedding}

We denote by $\mathcal{H}(A_a)$ the completion of $C_c^\infty(-a,a)$ with respect to 
$(Q_W|_{L^2(-a,a)})(v)
=Q_W^{\,a}(v)=\langle A_a v, v \rangle_{L^2}$. 
Since $U([v])\neq 0$ for every nonzero $v\in C_c^\infty(-a,a)$, 
the space $C_c^\infty(-a,a)$ embeds injectively into $\mathcal{H}(A_\infty)$ via $v \mapsto [v]$ 
as in the case of $C_c^\infty(\mathbb{R})$. 
This embedding extends to an injective map
\[
\mathcal{H}(A_a) ~\hookrightarrow~\mathcal{H}(A_\infty).
\]
Through this embedding, we may regard $\D_a$ and $\overline{\D}_{a,\theta}$ as operators on $\mathcal{H}(A_\infty)$, 
although they are no longer densely defined in $\mathcal H(A_\infty)$. 
However, we still have
\[
\D_a
=
\left.\D_\infty\right|_{C_c^\infty(-a,a)}
\]
\[
\bar{\D}_\infty~\subset~\mathsf{D}_{\pi/2}=\overline{\D}
\]
and thus, as $a \to \infty$, it is natural to expect that $\overline{\D}_{a,\theta}$ approximates the self-adjoint extension $\mathsf{D}_{\pi/2}=\overline{\D}$ of $\bar{\D}_\infty$ 
in a suitable sense. 
Since the spectrum of $\mathsf{D}_{\pi/2}$ is given by $\Gamma$, one further expects that 
the eigenvalues of $\overline{\D}_{a,\theta}$ approximate $\Gamma$ as $a \to \infty$ 
 after removing the contribution from $\mathcal N$. 

\subsection{Factorization of the Weil quadratic form}

By \eqref{EQ_704} and the identification of $U$ given above, we have
\[
Q_W(v)
=
\|U([v])\|_{\mathcal K}^2
=
\frac{1}{\pi}
\|\widehat{\mathcal{P}_D}v\|_{L^2}^2
=
\frac{1}{\pi}
\|\widehat{\mathcal{P}}Dv\|_{L^2}^2
\]
for $v\in C_c^\infty(\mathbb R)$, assuming RH.
Thus the Weil quadratic form admits a factorization through
$\widehat{\mathcal P}D$.
This identity should be understood as a factorization of the quadratic form,
rather than as an identity between operators on $L^2(\mathbb R)$.
Indeed, $A_\infty$ was introduced only as a formal symbol, whereas
$U:\mathcal H(A_\infty)\to\mathcal K$ is unitary and therefore
$U^\ast U=I$ on $\mathcal H(A_\infty)$.

There is a corresponding kernel-level interpretation.
Using the notation of \cite{Su25}, the function
$\mathfrak S_x$ belongs to $L^2(\mathbb R)$ for every $x\in\mathbb R$
without assuming RH.
Under RH, \cite[Theorem 4.2]{Su25} gives
\[
G_g(x,y)=\frac{1}{\pi}
\int_{-\infty}^{\infty}
\mathfrak S_x(z)\overline{\mathfrak S_y(z)}\,dz,
\]
where 
$
G_g(x,y)
=
g(x-y)-g(x)-g(-y)+g(0)$. 
This identity may be viewed as the kernel-level counterpart of the
quadratic-form factorization above.

Both $g$ and $\mathfrak{S}_x$ admit explicit descriptions that do not involve
the zeros of $\xi(1/2-iz)$; see \eqref{EQ_103} for $g$ and
\cite[(1.5), (1.6)]{Su25} for $\mathfrak{S}_x$.
(We note that in \cite{Su25}, the factor $i$ in (1.5) should be replaced with $-i$,
and the definition $\mathfrak{S}_t(z):=\mathfrak{S}_{-t}(z)$
after (1.6) should be replaced with
$\mathfrak{P}_t(z):=\overline{\mathfrak{P}_{-t}(\bar z)}$.)
It is therefore natural to ask whether the preceding kernel identity can be
established unconditionally. 
If this were possible, then $G_g(x,y)$ would be nonnegative definite, since
the left-hand side is a Gram kernel in $L^2(\mathbb R)$.
Consequently, $g$ would be a screw function, and RH would follow from
\cite[Theorem 1.2]{Su23}. 
We do not pursue this question further here.

\subsection{Considerations on the right-hand side of \eqref{EQ_112}} \label{section_7_8} 
A similar argument for the boundary form of $\D_a^\ast$ in Section~\ref{section_6} 
can also be carried out for $\mathcal{H}(A_\infty)$. 
In what follows, we explain how such an argument suggests the function appearing 
in the right-hand side of \eqref{EQ_112}.
Let $K(w,z)$ denote the reproducing kernel of $\mathcal{B}$, namely,
$\langle F, K(w,\cdot) \rangle_{\mathcal{B}}=F(w)$.
Then
\[
K(w,z)= \frac{\overline{E(w)}E(z)-\overline{E^\sharp(w)}E^\sharp(z)}{2\pi i(\bar{w}-z)},
\quad
E(z)=\xi(1/2-iz)+\xi'(1/2-iz),
\]
since $\mathcal{B}$ is the de Branges space generated by $E$ (\cite[Theorem 1.1]{Su25}).
Unless $w$ is a common zero of $E$ and $E^\sharp$, 
the function $K(w,\cdot)$ is an eigenfunction of $M^\ast$
corresponding to the eigenvalue $\bar{w}$.
Indeed, for every $F\in\mathfrak D(M)$, we have
\[
\langle MF,K(w,\cdot)\rangle_{\mathcal{B}}
=
wF(w)
=
w\langle F,K(w,\cdot)\rangle_{\mathcal{B}}
=
\langle F,\bar{w}K(w,\cdot)\rangle_{\mathcal{B}}.
\]
Since $\mathfrak D(M)$ is dense in $\mathcal B$
(\cite[Section~6]{Su25}),
it follows that
$K(w,\cdot)\in\mathfrak D(M^\ast)$ and
$M^\ast K(w,\cdot)=\bar wK(w,\cdot)$.
In particular, $K(\mp i,\cdot)$ belong to the corresponding deficiency spaces ${\rm Ker}(M^\ast\mp i)$.
Furthermore, they have the same norm: 
\[
\|K(\mp i,\cdot)\|_{\mathcal{B}}^2
=
\langle K(\mp i,\cdot),K(\mp i,\cdot)\rangle_{\mathcal{B}}
=
K(\mp i,\mp i)
=
\xi(3/2)\xi'(3/2)/\pi.
\]

Let $W_\theta(t):=K(-i,t)+e^{i\theta}K(i,t)$. 
For the boundary form
$W(F,G):=\langle M^\ast F,G\rangle_{\mathcal{B}}
-\langle F,M^\ast G\rangle_{\mathcal{B}}$, we obtain 
\begin{equation} \label{EQ_705}
\aligned
\overline{W(K(z,\cdot),W_\theta)}
&=(z-i)\overline{\langle K(z,\cdot),K(-i,\cdot)\rangle_{\mathcal{B}}}
+e^{i\theta}(z+i)\overline{\langle K(z,\cdot),K(+i,\cdot)\rangle_{\mathcal{B}}}\\
&=(z-i)K(-i, z)
+e^{i\theta}(z+i)K(+i, z) \\
&= 
\frac{ie^{i\theta/2}}{\pi}
\bigl(
C_\theta E(z) 
- 
\overline{C_\theta} E^\sharp(z)
\bigr),
\endaligned
\end{equation}
where
$C_\theta=(e^{i\theta/2}E^\sharp(-i)+e^{-i\theta/2}E^\sharp(i))/2=\xi(3/2)\cos(\theta/2)
+i\xi'(3/2)\sin(\theta/2)$.
In particular, when $\theta=\pi$, the right-hand side becomes 
$-(2i/\pi)\,
\xi'(3/2)\cdot \xi(1/2-iz)$. 

Let $f_z$ be the function satisfying
$\widehat{f_z}(t)=K(\bar z,t)/E(t)$.
Then $f_z$ belongs to the subspace $V(0)$ ($\subset L^2(0,\infty)$) (\cite[Lemma~5.1]{Su25}). 
If $v\in V(0)$, then 
\[
\langle v, f_{z} \rangle_{A_\infty}
= 2 \langle v, f_{z} \rangle_{L^2}
= \frac{1}{\pi} \langle \widehat{v}, \widehat{f_{z}} \rangle_{L^2}
= \frac{1}{\pi} \langle E \widehat{v}, K(\bar{z}, \cdot) \rangle_{\mathcal{B}}
=\frac{1}{\pi} E(\bar{z})\widehat{v}(\bar{z}) 
\]
by \cite[Theorem~5.5]{Su25}.
Hence, the analogue of the vector $v_z$ in Section~\ref{section_6} 
is given by
\[
v_z:=\frac{\pi f_z}{E^\sharp(z)}
\]
when $E^\sharp(z)\ne0$, 
and the analogue 
of the vector $v_\pm$ in Section~\ref{section_6} 
is a constant multiple of $f_{\pm i}$, 
since $\|f_{+i}\|_{A_\infty}=\|f_{-i}\|_{A_\infty}$
by definition. 
Applying $\widehat{f_{\pm i}}(z)=K(\mp i,z)/E(z)$
 to \eqref{EQ_705}, we obtain
\[
(z - i) \widehat{f_{+i}}(z)
+ e^{i\theta}  (z + i) \widehat{f_{-i}}(z)
= -
\frac{e^{i\theta/2}}{\pi i}
\left(C_\theta-\overline{C_\theta} \frac{E^\sharp (z)}{E(z)}\right). 
\]
In particular,  we have 
\begin{equation} \label{EQ_706}
(z - i)  \widehat{v_{+}}(z)
- (z + i) \widehat{v_{-}}(z)
=
\frac{\xi(1/2-iz)}{E(z)}
\end{equation}
by taking $\theta=\pi$ and $v_\pm:=(\pi i/(2\xi'(3/2)))\,f_{\pm i}$. 
As we have seen, 
$W(K(\bar{z},\cdot),W_\theta)$ and $W(v_z,w_\theta)$ in Section~\ref{section_6_5} 
have a structural similarity. This similarity suggests the formula \eqref{EQ_112}. 
To further support the validity of \eqref{EQ_112}, 
 we examine the relation 
between $\mathcal{H}(A_a)$ and $\mathcal{H}(A_\infty)$ in more detail 
in the next subsection. 

\subsection{Heuristics for formula \eqref{EQ_112}} \label{section_7_9}

We first show that the left-hand side of formula \eqref{EQ_112} 
is also related to de Branges spaces, 
and further examine the validity of this formula. 
We consider the space consisting of Fourier transforms  
$\widehat{\mathcal{H}(A_a)}=\{\widehat{v} \mid v \in \mathcal{H}(A_a)\}$, 
and define the inner product by
$
\langle \widehat{u}, \widehat{v} \rangle_{\widehat{A}_a}
:=\langle u, v \rangle_{A_a}. 
$
By definition, $\widehat{\mathcal{H}(A_a)}$ is a Hilbert space consisting of entire functions. 
Since the Fourier transform on $\mathcal{H}(A_a)$ is continuous, the point evaluation
 $\widehat{v} \mapsto \widehat{v}(z)$
is continuous on $\widehat{\mathcal{H}(A_a)}$ for every $z \in \C$. 
Suppose that 
(i) if $\widehat{v} \in \widehat{\mathcal{H}(A_a)}$, then
$\widehat{v}^\sharp(z):=\overline{\widehat{v}(\bar z)}$
also belongs to $\widehat{\mathcal{H}(A_a)}$ and
$\|\widehat{v}\|_{\widehat{A}_a}=\|\widehat{v}^\sharp\|_{\widehat{A}_a}$,
and moreover, 
(ii) if $\widehat{v} \in \widehat{\mathcal{H}(A_a)}$
and $\widehat{v}(w)=0$ for $w \in \C\setminus\R$, 
then $\widehat{v}_w(z):=((z-\bar{w})/(z-w))\widehat{v}(z)$
also belongs to $\widehat{\mathcal{H}(A_a)}$ and
$\|\widehat{v}\|_{\widehat{A}_a}=\|\widehat{v}_w\|_{\widehat{A}_a}$. 
Then 
$\widehat{\mathcal{H}(A_a)}$ is a de Branges space 
by the axioms of de Branges spaces. 
These properties (i) and (ii) follow from the fact that
$\widehat{L^2(-a,a)}$ is a Paley--Wiener space and 
the form-core property established in Section~\ref{section_3}, together with the observation that
\begin{equation} \label{EQ_707}
\| \widehat{v} \|_{\widehat{A}_a}^2
= \sum_{\gamma \in \Gamma} m_\gamma |\widehat{v}(\gamma)|^2
\end{equation}
holds by definition. Hence, $\widehat{\mathcal{H}(A_a)}$ is a de Branges space.
\medskip

As in Section~\ref{section_6_3},
take $v_z\in\mathcal{H}(A_a)$ satisfying
$
\langle v,v_z\rangle_{A_a}
=
\widehat v(\bar z)$ 
for all $v\in\mathcal{H}(A_a)$.  
Then
\[
J_a(w,z):=\langle v_{\bar w}, v_{\bar z} \rangle_{A_a} = \widehat{v_{\bar w}}(z) = \overline{\widehat{v_{\bar z}}(w)}
\]
is the reproducing kernel of $\widehat{\mathcal{H}(A_a)}$. Indeed,
\[
\langle \widehat{v}, J_a(w,\cdot) \rangle_{\widehat{A}_a}
=
\langle \widehat{v}, \widehat{v_{\bar w}} \rangle_{\widehat{A}_a}
=
\langle v, v_{\bar w} \rangle_{A_a}
= \widehat{v}(w)
\]
for every $\widehat{v} \in \widehat{\mathcal{H}(A_a)}$. 
Note that $\widehat{v_{w}}=J_a(\bar w, \cdot)$ by the uniqueness of the reproducing kernel.  
Since $\widehat{\mathcal{H}(A_a)}$ is a de Branges space, 
there exists an entire function $E_a$ in the Hermite--Biehler class such that
\[
J_a(w,z)= \frac{\overline{E_a(w)}E_a(z)-\overline{E_a^\sharp(w)}E_a^\sharp(z)}{2\pi i(\bar{w}-z)}.
\]
It follows that, for $w \in \C\setminus \R$,
\[
\| v_{w} \|_{A_a}^2=J_a(\bar w, \bar w)=J_a(w,w) = \| v_{\bar w} \|_{A_a}^2. 
\]
Therefore, after adjusting their phases if necessary, we may choose the normalized deficiency vectors $v_\pm$ in \eqref{EQ_602} so that
$v_\pm=c\, v_{\pm i}$ for some common $c\ne0$.
Thus, we obtain
\begin{equation} \label{EQ_708}
\aligned 
W(a,\theta;z) 
&= \overline{W(v_{\bar z}, w_\theta)} \\
& =
c (z - i)  J_a(-i, z)
+c  (z + i) e^{i\theta} J_a(+i, z) \\
& =  -\frac{c e^{i\theta/2}}{\pi i} \left(  C_{a,\theta}E_a(z) -\overline{C_{a,\theta}}E_a^\sharp(z) \right), 
\endaligned 
\end{equation}
where
$C_{a, \theta}=(e^{i\theta/2}E_a^\sharp(-i)+e^{-i\theta/2}E_a^\sharp(i))/2$. 
The similarity between this formula and \eqref{EQ_705} should be apparent.
Moreover, by comparing the norms \eqref{EQ_704} and \eqref{EQ_707},
and noting that $\| \widehat{v} \|_{\widehat{A}_a}^2=\| v \|_{A_a}^2=Q_W(v)$, 
we expect that the de Branges space $\widehat{\mathcal{H}(A_a)}$
approaches the de Branges space $\mathcal{B}$ as $a \to \infty$. 
(More precisely, the limiting space should be the space 
$\mathcal{K}=\mathcal{B}/E$.) 
Here, the convergence of these spaces is understood only heuristically, 
and no precise notion of convergence is specified.
This should be reflected in the relation between \eqref{EQ_708} and \eqref{EQ_705}.
Since \eqref{EQ_706} is obtained from \eqref{EQ_705} by specializing to $\theta=\pi$,
we thus anticipate that the relation between \eqref{EQ_708} and \eqref{EQ_705} 
in the limit $a \to \infty$ 
will be expressed by a limiting formula of the form \eqref{EQ_112}.

On the other hand, 
the reproducing kernel $J_a(z,w)$ of $\widehat{\mathcal{H}(A_a)}$ 
may develop a singularity in the limit $a \to \infty$, reflecting the degeneration caused by
$
\mathcal{N} = \{ v  \mid Q_W(v) = 0 \}
$ 
in the passage from $\mathcal{H}(A_a)$ to $\mathcal{H}(A_\infty)$.  
In fact,  we have 
\[
J_a(z,z) = \|v_{\bar z}\|_{A_a}^2 
= \sup_{0 \ne v \in  \mathcal{H}(A_a)} \frac{|\widehat{v}(z)|^2}{Q_W^a(v)}
\]
and hence, if there exists a sequence $v_a \in \mathcal{H}(A_a)$ such that 
$Q_W^a(v_a) \to 0$ and $\widehat{v_a}(z) \to c \ne 0$ as $a \to \infty$, then $J_a(z,z) \to \infty$. 
Thus, $J_a(z,w)$ may become unbounded on the diagonal in the limit $a \to \infty$. 
It is not clear whether $J_a(z,w)$ also develops a singularity off the diagonal in this limit. 
It remains to be clarified how the singularity on the diagonal
and any possible singularity off the diagonal are reflected 
in the relation between the boundary forms
$W(v_z,w_\theta)$ and
$W(K(\bar z,\cdot),W_\theta)$. 
Such a study would be necessary to determine the precise form of the limiting formula \eqref{EQ_112}. 
However, we do not pursue this issue further in the present paper, and leave it for future study.

%
\section{From distribution kernels to continuous kernels} \label{section_8}
%

\subsection{Passing from $A_a$ to $G_a$}

We have developed the theory of $Q_W^{\,a}$ and $A_a$
using the screw function associated with $\zeta(s)$.
However, in the theory of screw functions, as presented 
for example in \cite[Section 5]{KrLa14},
it is more natural to consider the integral operator
$\widetilde{G}$ defined by
\begin{equation} \label{EQ_801}
(\widetilde{G}u)(x):=\int_{-\infty}^{\infty}(g(x-y)-g(x)-g(-y)+g(0))u(y)\,dy
\end{equation}
rather than the integral operator $G$ in \eqref{EQ_104}.
As shown in \cite[Theorem 1.3]{Su23}, if
$Q_{\widetilde{G}}(u):= \langle \widetilde{G}u, u \rangle_{L^2}$
denotes the quadratic form associated with $\widetilde{G}$,
then RH is equivalent to the condition 
$Q_{\widetilde{G}}(u)\geq0 $
for every $a>0$ and every
$u\in C_c^\infty(\R)\cap L_0^2(-a,a)$.
This follows from the relation
$Q_W(v)=Q_{\widetilde{G}}(Dv)$,
$v\in C_c^\infty(\R)$
\cite[Proposition 3.1]{Su23}.
Note that if $u=Dv$ with
$v\in C_c^\infty(-a,a)$, then $u\in L_0^2(-a,a)$.
If $Q_G(v):= \langle Gv, v \rangle_{L^2}$
denotes the quadratic form associated with the integral operator $G$
in \eqref{EQ_104},
then $Q_{\widetilde{G}}(u)=Q_G(u)$ for every
$u \in L_0^2(-a,a)$,
and hence
$Q_W(v)=Q_{\widetilde{G}}(Dv)=Q_G(Dv)$.
Motivated by this relation, we chose to study $Q_W$
through the quadratic form $Q_G$,
whose kernel is considerably simpler.

\medskip

To represent $Q_W$ directly on the finite interval $(-a,a)$,
one must use the operator $A_a$ as in \eqref{EQ_101}.
However, if one is interested only in the positivity of $Q_W$,
there is no essential difference in working with $G_a$,
since $Q_W(v)=Q_G(Dv)$ and the restriction of $Q_G$ to $(-a,a)$
is represented by the operator $G_a$ through
$(Q_G|_{L_0^2(-a,a)})(u)=\langle G_a u, u\rangle_{L^2}$.
Moreover, $G_a$ is an integral operator with a continuous kernel
and is therefore much easier to handle than the unbounded operator $A_a$.
As we shall see below,
the Hilbert space obtained by completing with respect to
$\langle A_a u,u\rangle_{L^2}$
and the Hilbert space obtained by completing with respect to
$\langle G_a u,u\rangle_{L^2}$
are related through the identity
$Q_W(v)=Q_G(Dv)$,
and Theorem~\ref{thm_5}
can be interpreted within the framework of $G_a$.

Although $A_a$ and $G_a$ are very different as operators,
the former being an unbounded operator with discrete spectrum
accumulating at $+\infty$
and the latter being a compact operator with discrete spectrum
accumulating at $0$,
the situation is analogous from the viewpoint of positivity of quadratic forms,
since in both cases the relevant quantity is the bottom of the spectrum.
(There is, of course, the distinction that the infimum may or may not be attained.)
Furthermore, as we shall see below,
this difference does not affect the spectrum of the corresponding
self-adjoint extensions of the differential operator.

\medskip

To discuss these matters, we first review the relation between
the norms arising through the differentiation operator $D$.

\subsection{The Inverse Neumann Laplacian} \label{section_8_2}

The operator
$D:H_0^1(-a,a) \to L_0^2(-a,a)$
is bijective.
Therefore, if $u=Dv$ with $v \in H_0^1(-a,a)$, then $v=D^{-1}u$, 
and hence 
$
\|v\|_{L^2}^2
=
\langle D^{-1}u,D^{-1}u\rangle_{L^2}.
$ 
The operator $DD^\ast$ is the Neumann Laplacian on $L^2(-a,a)$.
Its restriction to $L_0^2(-a,a)$ is the positive self-adjoint operator
acting as $-d^2/dx^2$
with domain
\[
\{u\in H^2(-a,a)\cap L_0^2(-a,a)
\mid u'(a)=u'(-a)=0\}.
\]
In what follows, we denote this restriction by $-\Delta_N$.
It is bijective on $L_0^2(-a,a)$, and its inverse
\[
K_a:=(-\Delta_N)^{-1}:
L_0^2(-a,a)\longrightarrow L_0^2(-a,a)
\]
is a compact, self-adjoint, and positive integral operator with kernel
\[
N(x,y)
=
\frac{x^2+y^2}{4a}
-\frac{|x-y|}{2}
+\frac{a}{6}.
\]
Since $K_a=(DD^\ast)^{-1}$ on $L_0^2(-a,a)$, we have 
$
\|v\|_{L^2}^2
=
\langle K_a u,u\rangle_{L^2}.
$
On the other hand, if $u=Dv$ with $v\in H_0^1(-a,a)$, then
\[
v(x)
=
-i\int_{-a}^x u(t)\,dt,
\]
and therefore
\[
\aligned
\|v\|_{L^2}^2
&=
\int_{-a}^a
\left|
\int_{-a}^x u(t)\,dt
\right|^2dx
=
\int_{-a}^a\int_{-a}^a
(a-\max(t,s))u(t)\overline{u(s)}\,dt\,ds.
\endaligned
\]
Accordingly, defining
\[
(\widetilde K_a u)(x)
:=
\int_{-a}^a(a-\max(x,y))u(y)\,dy,
\]
we obtain
$
\|v\|_{L^2}^2
=
\langle\widetilde K_a u,u\rangle_{L^2}.
$
The operator $\widetilde K_a$ is compact, self-adjoint, and positive on
$L^2(-a,a)$.
Moreover,
\begin{equation} \label{EQ_802}
\|v\|_{L^2}^2
=
\langle\widetilde K_a u,u\rangle_{L^2}
=
\langle K_a u,u\rangle_{L^2},
\qquad
u=Dv\in L_0^2(-a,a).
\end{equation}
Thus $\widetilde K_a$ and $K_a$ induce the same quadratic form on
$L_0^2(-a,a)$.
In what follows, we use $K_a$, whose operator-theoretic meaning on
$L_0^2(-a,a)$ is more transparent.

\subsection{The Hilbert Space $\mathcal H(S_a)$}

Let $\lambda<\lambda_a$ and define
\[
S_a:=G_a-\lambda K_a:
L_0^2(-a,a)\longrightarrow L_0^2(-a,a).
\]
Since $G_a$ and $K_a$ are bounded self-adjoint operators on
$L_0^2(-a,a)$, so is $S_a$.
Moreover, if $u=Dv$ with $v\in H_0^1(-a,a)$, then
\[
\aligned
\langle S_a u,u\rangle_{L^2}
&=
\langle G_aDv,Dv\rangle_{L^2}
-\lambda\langle K_aDv,Dv\rangle_{L^2}
\\
&=
Q_W^{\,a}(v)-\lambda\|v\|_{L^2}^2
\geq
(\lambda_a-\lambda)\|v\|_{L^2}^2.
\endaligned
\]
Since $D:H_0^1(-a,a)\to L_0^2(-a,a)$ is bijective, it follows that
\[
\|u\|_{S_a}^2
:=
\langle S_a u,u\rangle_{L^2}
\]
defines a norm on $L_0^2(-a,a)$.
Let
\[
\mathcal H(S_a)
:=
\Bigl[
\text{the completion of
$C_c^\infty(-a,a)\cap L_0^2(-a,a)$
with respect to $\|\cdot\|_{S_a}$}
\Bigr].
\]
For $v\in C_c^\infty(-a,a)$, we have
$
\|Dv\|_{S_a}=\|v\|_{T_a}.
$
Since
$
D:
C_c^\infty(-a,a)
\longrightarrow
C_c^\infty(-a,a)\cap L_0^2(-a,a)
$
is bijective, $D$ extends uniquely to a unitary isomorphism
\[
\bar D:
\mathcal H(T_a)
\overset{\sim}{\longrightarrow}
\mathcal H(S_a).
\]

Through the same unitary isomorphism, define
\[
\widetilde{\D}_a
:=
\bar D\,\D_a\,\bar D^{-1}
\]
on $\mathcal H(S_a)$.
Its domain is
\[
\mathfrak D(\widetilde{\D}_a)
=
D\bigl(C_c^\infty(-a,a)\bigr)
=
C_c^\infty(-a,a)\cap L_0^2(-a,a).
\]
On this domain $\bar D$ agrees with $D$, and hence
\[
\widetilde{\D}_a
=
D\D_aD^{-1}.
\]
For $u\in C_c^\infty(-a,a)\cap L_0^2(-a,a)$,
\[
(D^{-1}u)(x)
=
-i\int_{-a}^x u(t)\,dt,
\]
and therefore
\[
D\D_a(D^{-1}u)(x)
=
iu'(x).
\]
Thus $\widetilde{\D}_a$ acts as $i\,d/dx$ on its domain.
By unitary equivalence,
\[
\widetilde{\D}_a^\ast
=
\bar D\,\D_a^\ast\,\bar D^{-1}.
\]
Consequently, $\widetilde{\D}_a$ is symmetric with deficiency indices
$(1,1)$, and its self-adjoint extensions are unitarily equivalent to those
of $\D_a$ and have the same spectra.

For
\[
u_z:=\bar Dv_z,\qquad
u_\pm:=\bar Dv_\pm,\qquad
\widetilde w_\theta:=u_++e^{i\theta}u_-,
\]
we have
\[
\widetilde{\D}_a^\ast u_z=zu_z,
\qquad
\widetilde{\D}_a^\ast u_\pm=\pm i u_\pm.
\]
Define the corresponding boundary form by
\[
\widetilde W(u,v)
:=
\langle\widetilde{\D}_a^\ast u,v\rangle_{S_a}
-
\langle u,\widetilde{\D}_a^\ast v\rangle_{S_a}.
\]
Since $\bar D$ is unitary,
\[
\widetilde W(\bar D u,\bar D v)=W(u,v)
\]
for all $u,v\in\mathfrak D(\D_a^\ast)$.
In particular,
\begin{equation} \label{eq_260711_3}
\widetilde W(u_z,\widetilde w_\theta)
=
(z+i)\langle u_z,u_+\rangle_{S_a}
+
e^{-i\theta}(z-i)\langle u_z,u_-\rangle_{S_a}.
\end{equation}
Moreover,
\[
\langle u_z,u_\pm\rangle_{S_a}
=
\langle v_z,v_\pm\rangle_{T_a}
=
\overline{\widehat{v_\pm}(\bar z)}.
\]
It follows from \eqref{EQ_111} that
\[
W(a,\theta;z)
=
\overline{W(v_{\bar z},w_\theta)}
=
\overline{\widetilde W(u_{\bar z},\widetilde w_\theta)}.
\]
Thus the boundary forms determining the eigenvalues are preserved under
the unitary equivalence between $\mathcal H(T_a)$ and
$\mathcal H(S_a)$.

\medskip

Although $K_a$ is defined as an operator on $L_0^2(-a,a)$, 
the same continuous symmetric kernel $N(x,y)$ defines a compact self-adjoint
integral operator on $L^2(-a,a)$. 
In what follows, we use $N(x,y)$ only as an integral kernel, 
without imposing a zero-mean condition on $v_\pm$. 
Let
\[
k(x,y):=g(x-y)-\lambda N(x,y).
\]
Since 
$
-\partial_x^2N(x,y)
=
\delta(x-y)-1/(2a)$ 
in the sense of distributions, we have
\[
-\partial_x^2k(x,y)
=
-g''(x-y)-\lambda\delta(x-y)+\frac{\lambda}{2a}.
\]
Assuming the form-core approximation used in Section~\ref{section_3},
the polarized form of \eqref{EQ_209}, together with the definition of $A_a$,
gives
$
A_av
=
\int_{-a}^{a}(-g''(x-y))v(y)\,dy
$
in the sense of distributions on $(-a,a)$
for $v\in\mathfrak D(A_a)$.
Hence the equations
$
T_av_\pm=C_\pm e_{\pm i}
=
C_\pm e^{\pm x}$ 
give
\begin{equation} \label{eq_260725_1}
\int_{-a}^{a}
(-k_{xx}(x,y))v_\pm(y)\,dy
=
C_\pm e^{\pm x}
+
\frac{\lambda}{2a}
\int_{-a}^{a}v_\pm(y)\,dy
\end{equation}
in the sense of distributions on $(-a,a)$. 
Since $k(x,y)$ is continuous and
$v_\pm\in L^2(-a,a)\subset L^1(-a,a)$, the function
$
x\longmapsto
\int_{-a}^{a}k(x,y)v_\pm(y)\,dy
$
is continuous.
Equation \eqref{eq_260725_1} then implies that this function is smooth on
$(-a,a)$.
Integrating twice, we obtain constants $A_\pm,B_\pm\in\C$ such that
\begin{equation} \label{eq_260725_2}
\int_{-a}^{a}k(x,y)(-v_\pm(y))\,dy
=
C_\pm e^{\pm x}
+
\frac{\lambda x^2}{4a}
\int_{-a}^{a}v_\pm(y)\,dy
+
A_\pm x+B_\pm ,
\quad x\in(-a,a).
\end{equation}
Equivalently,
\begin{equation} \label{eq_260725_3}
\int_{-a}^{a}
\left(
k(x,y)+\frac{\lambda x^2}{4a}
\right)
(-v_\pm(y))\,dy
=
C_\pm e^{\pm x}+A_\pm x+B_\pm .
\end{equation}
Thus \eqref{eq_260725_3} is a Fredholm integral equation of the first kind
with a continuous kernel, together with the two scalar unknowns
$A_\pm$ and $B_\pm$.

Here we only use the fact that
$T_av_\pm=C_\pm e_{\pm i}$ implies
\eqref{eq_260725_2} and \eqref{eq_260725_3}.
We do not claim at this point that an arbitrary $L^2$-solution of
\eqref{eq_260725_3} necessarily belongs to $\mathfrak D(T_a)$ and satisfies
$T_av_\pm=C_\pm e_{\pm i}$.
Establishing such a converse requires an additional domain argument. 
Equation \eqref{eq_260725_3} nevertheless provides a concrete starting point
for a numerical approximation of $v_\pm$, after fixing the normalization
and treating $A_\pm$ and $B_\pm$ as additional unknowns.

\subsection{Relation to Bombieri's Problem 1} \label{Bombieri_Problem_1}

Using $\widetilde{G}$ in \eqref{EQ_801}
and $P_a$ in \eqref{EQ_201}, define
\[
\widetilde{G}_a:=P_a \widetilde{G}P_a : L_0^2(-a,a) \to L_0^2(-a,a).
\]
Recall that $G_a=P_aG P_a$ in \eqref{EQ_105},
and $B_a=D^\ast G_a D$ in \eqref{EQ_106}.
Then
\begin{equation} \label{EQ_803}
\langle \widetilde{G}_a Dv, Dv \rangle_{L^2}
=
\langle G_a Dv, Dv \rangle_{L^2}
=
\langle B_a v, v \rangle_{L^2}
=
Q_W(v),
\quad v \in C_c^\infty(-a,a),
\end{equation}
see \cite[Proposition 3.1]{Su23}.
By density of $C_c^\infty(-a,a)$ in $H_0^1(-a,a)$
and continuity of both sides with respect to the $H^1$-norm,
the identity $Q_W(v)=\langle G_aDv,Dv\rangle_{L^2}$ 
extends to every $v\in H_0^1(-a,a)$. 
For finite $a$, the norm on $H_0^1(-a,a)$ defined by
$\|v\|_{L^2}+\|v'\|_{L^2}$ is equivalent to that defined by
$\|v'\|_{L^2}$.
In this subsection, we adopt the latter norm and write 
$\|v\|_{H^1}:=\|v'\|_{L^2}$, $v\in H_0^1(-a,a)$.
Then the linear map
$H_0^1(-a,a)\to L_0^2(-a,a)$, 
$v\mapsto Dv$,
is an isometric isomorphism.
Thus, by \eqref{EQ_803},
\[
\frac{Q_W(v)}{\|v\|_{H^1}^2}
=
\frac{Q_G(u)}{\|u\|_{L^2}^2},
\qquad 
u=Dv.
\]
Hence the corresponding variational problem on $H_0^1(-a,a)$
is transformed into the ordinary Rayleigh quotient problem
for the compact self-adjoint operator $G_a$ on $L_0^2(-a,a)$.
This variational problem is essentially of the same type as
\cite[Problem 1]{Bo01}
(and \cite[Problem B]{Bo03}),
where Weil's quadratic form is minimized on a Sobolev space.

Equivalently, since $Q_W$ is bounded with respect to the
$H^1$-norm adopted above, the Riesz representation theorem
shows that the associated sesquilinear form is represented
by a bounded self-adjoint operator on $H_0^1(-a,a)$.
Under the isometric isomorphism
$D:H_0^1(-a,a)\to L_0^2(-a,a)$,
this operator is unitarily equivalent to
$\widetilde{G}_a=G_a$.
Thus the present formulation realizes the Sobolev-space variational problem
as an ordinary compact self-adjoint eigenvalue problem on $L_0^2(-a,a)$.

For comparison, \cite[Theorem 4]{Bo01} asserts that, in Bombieri's
original variational problem, the infimum is attained if it is negative.
In the present formulation, the corresponding assertion is immediate:
if the infimum of the Rayleigh quotient for $G_a$ is negative, then it is
a nonzero spectral value of the compact self-adjoint operator $G_a$,
and hence an eigenvalue.
\medskip

\subsection{Reformulation as a generalized eigenvalue problem}~
By \eqref{EQ_802},
$
\|v\|_{L^2}^2
=
\langle K_a u,u \rangle_{L^2}
$
with $u=Dv$ and $K_a=(-\Delta_N)^{-1}$.
Combining this with \eqref{EQ_803},
the Rayleigh quotient in \eqref{EQ_107} can be rewritten as
\begin{equation} \label{EQ_804}
\frac{Q_W^{\,a}(v)}{\|v\|_{L^2}^2}
=
\frac{\langle G_a u, u \rangle_{L^2}}
{\langle K_a u, u \rangle_{L^2}}
\end{equation}
for $v \in H_0^1(-a,a)$ and $u=Dv \in L_0^2(-a,a)$.
Thus, by Corollary~\ref{cor_2},
\[
\lambda_a
=
\inf_{0\ne u\in L_0^2(-a,a)}
\frac{\langle G_a u,u\rangle_{L^2}}
{\langle K_a u,u\rangle_{L^2}},
\]
which may be regarded as a generalized Rayleigh quotient
for the pair $(G_a,K_a)$ on $L_0^2(-a,a)$.
Formally, the corresponding generalized eigenvalue equation is
\[
G_a u=\mu K_a u.
\]
A rigorous formulation of the full generalized eigenvalue problem
requires interpreting this equation in the appropriate completed
Hilbert space.
(The case $\mu=0$ is exceptional:
then the contribution of $K_a$ disappears,
and one is simply looking at the kernel
(the $0$-eigenspace) of $G_a$.)

Although $A_a$ is unbounded,
both $G_a$ and $K_a$ are compact operators
with continuous integral kernels.
In the generalized eigenvalue problem,
one must analyze not only the spectra of $G_a$ and $K_a$
individually but also the interaction between them.
Thus it cannot be said a priori that the spectral analysis of $A_a$
becomes simpler.
Nevertheless, the availability of compact-operator techniques
may provide certain advantages.

\bigskip

\noindent
{\bf Acknowledgment} 
\medskip

\noindent
This work was supported by JSPS KAKENHI Grant Number 
JP23K03050. 
This work was also supported by the Research Institute for Mathematical Sciences, 
an International Joint Usage/Research Center located in Kyoto University.

%

%
\bigskip 

\noindent
Department of Mathematics, \\
School of Science, \\ 
Institute of Science Tokyo \\
2-12-1 Ookayama, Meguro-ku, \\
Tokyo 152-8551, Japan  \\[2pt]
Email: {\tt msuzuki@math.sci.isct.ac.jp}


\begin{thebibliography}{99}
%
\bibitem{Bo01}
E. Bombieri,
\newblock{Remarks on Weil's quadratic functional in the theory of prime numbers. I}, 
\newblock{\it Atti Accad. Naz. Lincei Cl. Sci. Fis. Mat. Natur. Rend. Lincei (9) Mat. Appl.} 
\newblock{{\bf 11} (2000), no. 3, 183--233 (2001)}. 
%
\bibitem{Bo03}
E. Bombieri,
\newblock{A variational approach to the explicit formula}, 
\newblock{\it Comm. Pure Appl. Math. } 
\newblock{{\bf 56} (2003), no. 8, 1151--1164}. 
%
\bibitem{CC23}
A. Connes, C. Consani, 
\newblock{Spectral triples and $\zeta$-cycles}, 
\newblock{\it Enseign. Math.} 
\newblock{{\bf 69} (2023), no. 1-2, 93--148}. 
%
\bibitem{CCM25}
A. Connes, C. Consani, H. Moscovici,
\newblock{Zeta spectral triples},
\newblock{in {\it Applications of Noncommutative Geometry to Gauge Theories,
Field Theories, and Quantum Space-Time}},
\newblock{EMS Ser. Lect. Math. {\bf 37} (2026)}.
%
\comment{
\bibitem{CCM25}
A. Connes, C. Consani, H. Moscovici, 
\newblock{Zeta spectral triples}, 
\newblock{\url{https://arxiv.org/abs/2511.22755}} 
}
%
\bibitem{FOT11}
M. Fukushima, Y. Oshima, M. Takeda, 
\newblock{Dirichlet forms and symmetric Markov processes}, 
\newblock{Second revised and extended edition, De Gruyter Studies in Mathematics, {\bf 19}}, 
\newblock{\it Walter de Gruyter \& Co., Berlin}, 
\newblock{2011}. 
%
\bibitem{GrRy07}
I. S. Gradshteyn, I. M. Ryzhik, 
\newblock{Table of integrals, series, and products. Seventh edition}, 
\newblock{\it Elsevier/Academic Press, Amsterdam}, 
\newblock{2007}. 
%
\bibitem{Ka95}
T. Kato, 
\newblock{Perturbation theory for linear operators. Reprint of the 1980 edition. Classics in Mathematics}, 
\newblock{\it Springer-Verlag, Berlin}, 
\newblock{1995}. 
%
\bibitem{KrLa14}
M. G. Kre\u{\i}n, H. Langer, 
\newblock{Continuation of hermitian positive definite functions and related questions}, 
\newblock{\it Integral Equations Operator Theory}  
\newblock{{\bf 78} (2014), no. 1, 1--69}. 
%
\bibitem{LSV09}
D. Lenz, P. Stollmann, I. Veseli\'{c}, 
\newblock{The Allegretto-Piepenbrink theorem for strongly local Dirichlet forms}, 
\newblock{\it Doc. Math.}  
\newblock{{\bf 14} (2009), 167--189}. 
%
\bibitem{NIST}
F. Olver, D. Lozier, R. Boisvert, C. Clark (eds), 
\newblock{NIST handbook of mathematical functions}, 
\newblock{\it Cambridge University Press, Cambridge}, 
\newblock{2010}. 
%
\bibitem{RS75}
M. Reed, B. Simon, 
\newblock{Methods of modern mathematical physics. II. Fourier analysis, self-adjointness}, 
\newblock{\it Academic Press [Harcourt Brace Jovanovich, Publishers], New York-London}, 
\newblock{1975}. 
%
\bibitem{Sch12}
K. Schm\"{u}dgen,
\newblock{Unbounded self-adjoint operators on Hilbert space}, 
\newblock{Graduate Texts in Mathematics {\bf 265}}, 
\newblock{\it Springer, Dordrecht}, 
\newblock{2012}. 
%
\bibitem{Su23} 
M. Suzuki,
\newblock{Aspects of the screw function corresponding to the Riemann zeta function}, 
\newblock{\it J. Lond. Math. Soc.}  
\newblock{{\bf 108} (2023), no.4, 1448-1487}. 
%
\bibitem{Su25}
M. Suzuki,
\newblock{On the Hilbert space derived from the Weil distribution}, 
\newblock{\it Canadian Journal of Mathematics}  
\newblock{(FirstView)}. 
%
\bibitem{We52}
A. Weil,
\newblock{Sur les ``formules explicites'' de la th\'{e}orie des nombres
              premiers}, 
\newblock{\it Comm. S\'{e}m. Math. Univ. Lund [Medd. Lunds Univ. Mat. Sem.]},  
\newblock{{\bf 1952} (1952), Tome Suppl\'{e}mentaire, 252--265}. 
%
\bibitem{We72}
A. Weil,
\newblock{Sur les formules explicites de la th\'{e}orie des nombres}, 
\newblock{\it Izv. Akad. Nauk SSSR Ser. Mat.},  
\newblock{{\bf 36} (1972), 3--18.}. 
%
\bibitem{Yo92}
H. Yoshida,
\newblock{On Hermitian forms attached to zeta functions}, 
\newblock{\it Zeta functions in geometry (Tokyo, 1990)}, 
\newblock{281--325, Adv. Stud. Pure Math., 21}, 
\newblock{\it Kinokuniya, Tokyo}, 
\newblock{1992}.
%
\end{thebibliography}
\end{document}